\documentclass{amsart}
\usepackage[english]{babel}
\usepackage[dvips]{graphicx}
\usepackage{amssymb}
\usepackage{latexsym}
\usepackage{xypic}
\usepackage{multirow}

\newtheorem{theorem}{Theorem}[section]
\newtheorem{lemma}[theorem]{Lemma}
\newtheorem{proposition}[theorem]{Proposition}

\newtheorem{corollary}[theorem]{Corollary}

\theoremstyle{definition}
\newtheorem{definition}[theorem]{Definition}
\newtheorem{example}[theorem]{Example}

\theoremstyle{remark}
\newtheorem{remark}[theorem]{Remark}

\numberwithin{equation}{section}

\def\N{\bf{N}}

\def\C{{\mathbb C}}
\def\R{{\mathbb R}}
\def\Z{{\mathbb Z}}
\def\H{{\mathbb H}}

\def\P{\mathbb P_{\mathbb C}}

\def\PSL{\text{{PSL}}}

\def\PU{\text{{PU}}}

\begin{document}

\title{  \sc{Projective Cyclic Groups in Higher Dimensions } }

\author{Angel Cano}
\address{ UCIM, Av. Universidad s/n. Col. Lomas de Chamilpa,
C.P. 62210, Cuernavaca, Morelos, M\'exico.}
\email{angelcano@im.unam.mx}

\author{Luis Loeza}
\address{IIT UACJ, Av. del Charro 610 Norte,  Partido Romero, C.P. 32310, Ciudad Ju\'arez, Chihuahua, M\'exico}
\email{  luis.loeza@uacj.mx}

\author{Alejandro Ucan-Puc}
\address{UCIM, Av. Universidad s/n. Col. Lomas de Chamilpa,
C.P. 62210, Cuernavaca, Morelos, M\'exico.}
\email{manuel.ucan@im.unam.mx}

  \thanks{Supported by grants of the PAPPIT's  project IN101816 and CONACYT's 272169.} 

\subjclass{Primary 37F99, 32Q, 32M Secondary 30F40, 20H10, 57M60, 53C}



\begin{abstract}
In this article we provide  a  classification of the projective transformations  in $PSL(n+1,\C)$ considered as automorphisms of the complex projective space $\P^n$. Our classification is an interplay between  algebra and  dynamics.
 Just as in the case of isometries of $CAT(0)$-spaces, this is given by means of  three types of transformations, namely:   elliptic, parabolic and loxodromic. 
We    describe the  dynamics in each case, more precisely we determine  the corresponding  Kulkarni's limit set, the equicontinuity region, minimal sets, the discontinuity region and   maximal regions where the corresponding cyclic group acts properly discontinuously. We also  provide, in each case, some equivalent ways to classify the projective transformations.
\end{abstract}

\maketitle

\section*{Introduction}

Discrete groups of projective transformations arise as monodromy groups of ordinary differential equations (see \cite{iwa}),  associated to Ricatti 's foliations  (see \cite{santos}) or as the monodromy groups of the so called orbifold uniformizing differential equations (see \cite{yos}). However outside the groups coming from complex hyperbolic geometry,  little is known about their dynamics, see \cite{CSN}. Yet,  as in the one dimensional case, one might expect   interesting results. In this paper we deal with the basic problem of classifying the projective transformations and understand their dynamic.\\

When we look at  elements in  $\PU(1,n)$, one has that they preserve a ball, then, as in the one dimensional case, this fact   enables us to classify the transformations in  $\PU(1,n)$  by means  of their fixed points  and their position in the closed complex ball. More precisely, an element is  said to be:       elliptic if it has a fixed point  in the  complex ball,  parabolic if it  has a unique fixed point in the boundary of the complex ball  and finally the element is said to be   loxodromic if it has exactly  two fixed points in the boundary of the complex ball. However, when we deal with  automorphisms of $\P^n$, this type of classification  makes no sense, since  in general there is not an invariant ball.  
So, to extend the previous classification to $PSL(n+1,\C)$, we must think  dynamically, more precisely we must look into  the local behavior around  the  fixed points. The following definition captures this point of view.

\begin{definition} \label{d:tipos}
Let $\gamma\in PSL(n +1,\C)$ be a projective transformation, then 
\begin{enumerate}
\item The element $\gamma$ is called elliptic if there is a  lift $\widetilde\gamma\in SL(n+1,\C)$ of $\gamma$, which  is diagonalizable and each of its eigenvalues is unitary.
\item The element $\gamma$ is called  loxodromic is there is a  lift $\widetilde\gamma\in SL(n+1,\C)$ of $\gamma$ having  non-unitary  eigenvalue.
\item The element $\gamma$ is parabolic, if there is a 
 lift $\widetilde\gamma\in SL(n+1,\C)$ of $\gamma$ which is   non diagonalizable and has only unitary eigenvalues.
\end{enumerate}
\end{definition}

Clearly this definition exhausts all the possibilities and agrees with the standard classification in the one and    two dimensional setting,  as well as, in the case of  transformations in
 $PU(1,n)$, $n\geq 1$, see \cite{CSN, goldman}.  There is also another classification of the projective transformations of $PSL(3,\C)$ in terms of the fixed set given in \cite{santos}, which is closely related to our  classification but  properly talking does not agree with the one exposed here.

 As in the two  dimensional case, see \cite{CSN} ,  we want to know the relations between  different notions of limit set  (Kulkarni limit set, complement of the discontinuity set, complement of the  equicontinuity set, minimal sets, complements of  maximal sets where the action  is properly discontinuously) for complex Kleinain groups in all dimensions, and  a nice starting point towards the solution of this  problem  is  providing  a description of the after told  mentioned sets   for  cyclic groups. In this article we show:

\begin{theorem}[The discontinuity set] \label{t:l2}
Let $\gamma\in PSL(n+1,\C)$ be a projective transformation, then:
\begin{enumerate}
\item The element $\gamma$ is elliptic if and only if the set of accumulation  points of orbits of points in  $\P^n$ under the action of $\langle \gamma \rangle $  either is  empty or the whole space 
$\P^n$, depending on whether    $\gamma$ has finite order or not. 

\item The element $\gamma$ is parabolic if and only if the set of accumulation  points of orbits of points in  $\P^n$ under the action of $\langle \gamma \rangle $ is a single proper projective subspace {\rm (see Theorem \ref{l:lila} for a detailed description)}.

\item The element $\gamma$ is loxodromic if and only if the set of accumulation  points of orbits of points in  $\P^n$ under the action of $\langle \gamma \rangle $ is 
 a finite disjoint union of at least two projective subspaces {\rm (see Theorem \ref{l:lila} for a detailed description)}. 
 \end{enumerate}
\end{theorem}

\begin{theorem}[The equicontinuity set] \label{t:e2}
Let $\gamma\in PSL(n +1,\C)$ be a projective transformation, then:
\begin{enumerate}
\item The element $\gamma$ is elliptic if an only if
 the equicontinuity set of $\langle \gamma \rangle $ is the whole space $\P^n$.

\item The element $\gamma$ is parabolic if and only if
the  complement of the equicontinuity set of $\langle \gamma \rangle $  is a single  projective subspace $\mathcal{L}$ {\rm (see Theorem \ref{t:equi} for a precise description)}.

\item The element $\gamma$ is loxodromic if and only if
 the complement  of the  equicontinuity set of $\langle \gamma \rangle $  can be described as  the union of two proper distinct  projective subspaces $\mathcal{L}_1,\,\mathcal{L}_2$  of $\P^n$ {\rm ( see Theorem \ref{t:equi} for a precise description)}.

\end{enumerate}
\end{theorem}

The  Kurkarni's discontinuity set was introduced in \cite{kulkarni} as a way to construct  regions where a  group acts properly discontinuously   (see the formal definitions bellow  and see \cite{CSN, kulkarni} for a detailed discussion), respect this notion we show.

\begin{theorem}[The Kulkarni's limit set] \label{t:p2}
Let $\gamma\in PSL(n+1,\C)$ be a projective transformation, then:

\begin{enumerate}
\item The element $\gamma$ is elliptic if an only if   the  Kulkarni's limit set of  $\langle \gamma \rangle $ is either empty or the whole space 
$\P^n$, depending on whether    $\gamma$ has finite order or not.  In this case the Kulkarni discontinuity region of $\langle \gamma \rangle $ is the largest open set on which $\langle \gamma \rangle $ acts properly discontinuously.

\item The element $\gamma$ is parabolic if and only if
the  Kulkarni's limit  set of $\langle \gamma \rangle $   is a projective subspace $\mathfrak{L}$ {\rm (see Theorem \ref{l:par22} for a precise description)}.  In this case the Kulkarni discontinuity region of $\langle \gamma \rangle $ might not be is the largest open set on which $\langle \gamma \rangle $ acts properly discontinuously, see Corollary \ref{c:parmax}.

\item The element $\gamma$ is loxodromic if and only if the Kulkarni's limit set 
 of $\langle \gamma \rangle $  can be described as  the union of two proper distinct  projective subspaces $\mathfrak{L}_1,\,\mathfrak{L}_2$  of $\P^n$ {\rm ( see Theorem \ref{t:kpar2} for a precise description)}.  In this case the Kulkarni's discontinuity region of $\langle \gamma \rangle $ might not be is the largest open set on which $\langle \gamma \rangle $ acts properly discontinuously, see example \ref{e:loxmax}.
\end{enumerate}
\end{theorem}

Theorems \ref{t:e2} and \ref{t:p2}   looks-alike, indeed for $n\leq 2$ it is true that the equicontinuity set and the Kulkarni's discontinuity region agree, see \cite{CSN}, however as we well see in Examples \ref{e:ekdp1}, \ref{e:ekdp2}, \ref{e:ekdl1} and \ref{e:ekdl2}, it happens that in the higher dimensional case the discontinuity set and the  Kurkarni's  discontinuity region  might or might  not agree.

From the one and two dimensional settings we know that Definition \ref{d:tipos} can be given in terms of certain foliations (see \cite{CSN}), in  fact  this  provides a simple and useful  way  to  describe  the global dynamics of cyclic  groups. In this article we propose a generalization of such foliations,  
 but before we state the analogous results,  let us introduce some notation. Let $\C^{k,l}$, $0<k\leq l$, be a copy of $\C^{k+l}$ equipped  with the hermitian form: 
\[
\prec  U,V\succ _{k,l}=-\sum_{j=1}^k u_j\bar{v}_j+\sum_{j=k+1}^{k+l} u_j\bar{v}_j. 
\]
Here  $U(k,l)$ will denote  the subgroup of  $GL(k+l,\C)$ preserving $\prec , \succ _{k,l} $ and $PU(k,l)$ the respective projectivization. We define the $(k,l)$-ball, in symbols   $\H^{k,l}_{\C}$, as the projectivization of the set:

\[
\left \{
(z_1,\ldots, z_{k+l})
\in \C^{k+l}: \sum_{j=1}^k  \vert z_j\vert^2 >\sum_{j=k+1}^{k+l} \vert z_j\vert^2  
\right \}.
\]
 In a similar way the    $(k,l)$-sphere  is the projectivization  of the  set  
\[
\left \{
(z_1,\ldots, z_{k+l})
\in \C^{k+l}: \sum_{j=1}^k  \vert z_j\vert^2 =\sum_{j=k+1}^{k+l} \vert z_j\vert^2  
\right \}.
\]
Observe  the case $k=1$ correspond to  the context of the complex hyperbolic geometry and  in this case the (1,l)-sphere is  homeomorphic to the  $2l-1$-sphere and the $(1,l)$-ball is a model of the complex hyperbolic geometry.
With this in mind we finally state our results.

\begin{theorem} \label{t:e1}
Let $\gamma\in PSL(n +1,\C)$ be a projective transformation, then  $\gamma$ is elliptic if and only if  $\gamma$ preserves, up to conjugation, a foliation of $\C^n\setminus \{0\}$ by concentric $(1,n)$-spheres.
\end{theorem}

\begin{theorem} \label{t:p1}
Let $\gamma\in PSL(n+1,\C)$ be a projective transformation, then  $\gamma$ is parabolic if and only if and only if there are:

 \begin{enumerate}
 \item Natural numbers 
$0<k\leq l$ satisfying   $k+l=n+1$.
 \item A transformation  $\tau\in PSL(n+1,\C)$.
 \item      A non-empty family $\mathcal{F}$ of  $(k,l)$-spheres. 
\item  A projective  subspace $\mathcal{Z}\subset  \Bbb{P}^n_{\Bbb{C}}$.
\end{enumerate} 
satisfying
\begin{enumerate}
\item Each element  of $\mathcal{F}$ is  $\tau \gamma\tau^{-1}$-invariant.
\item The projective space $\mathcal{Z}$ is $\tau \gamma\tau^{-1}$-invariant.
\item For every pair of different elements
$T_{1},  T_{2}\in  \mathcal{F}$ one has $$\emptyset \neq \mathcal{ Z}\subset T_1\cap T_2\subset \mathcal{Z}^\bot. $$
\item  We have $ \bigcup_{T\in \mathcal{F}} T\setminus \mathcal{Z}^\bot=\P^n\setminus \mathcal{Z}^\bot$.
\item \label{c:pe} The fixed set satisfies: $Fix(\gamma)\subset \mathcal{Z}^\bot.$
\item \label{c:pl}The action of $\gamma$ restricted  to  $\mathcal{Z}^\bot $ is a given non-loxodromic. 
\end{enumerate}
\end{theorem}

Condition (\ref{c:pe}) 
 in the previous definition appear  bit artificial, but as one can see from Example \ref{e:parfix}, it is  necessary.

\begin{theorem} \label{t:l1}
Let $\gamma\in PSL(n+1,\Bbb{C})$ be a projective transformation, then  $\gamma$ is loxodromic if and only if
there is a proper   open set $W\subset \P^n$ such that  $\gamma (\overline  W)\subset  W$.
\end{theorem}

As corollary we get:

\begin{corollary} \label{c:cfix}
Let $\gamma\in PSL(n+1,\C)$ be a projective transformation, then 
\begin{enumerate}
\item The element $\gamma$ is loxodromic if and only if there are couple of distinct points $x,y\in Fix(\gamma):=\{z\in \P^n: \gamma z=z\}$ such that the action of $\gamma$ restricted to the complex line $\langle \langle x ,y\rangle \rangle$ is loxodromic.
\item The element $\gamma$ is parabolic  if and only if  every lift  $\widetilde \gamma\in SL(n+1,\C)$ is non-diagonalizable  and   for every 
 couple of distinct points $x,y\in Fix(\gamma)$  the action of $\gamma$ restricted to the complex line $\langle \langle x,y \rangle \rangle$ is elliptic.
\item The element $\gamma$  is elliptic  if and only if  every lift  $\widetilde \gamma\in SL(n+1,\C)$ is diagonalizable  and for every 
 couple of distinct points $x,y\in Fix(\gamma)$  the action of $\gamma$ restricted to the complex line $\langle \langle x,y\rangle \rangle$ is elliptic.
\end{enumerate}
\end{corollary}

Finally we prove the following result, concerning the existence of loxodromic elements in discrete groups.

\begin{theorem}
	Let $\Gamma\subset PSL(n+1,\C)$ be a strongly irreducible subgroup, then $\Gamma$ contains a loxodromic element.
	\end{theorem}

The paper is organized as follows: 
in Section \ref{s:recall}, we
review some general facts and introduce the notation used along the text, in Section \ref{s:equi},  
we describe the discontinuity of a group, Sections \ref{s:eli}, \ref{s:par} and \ref{s:lox} deals with 
description of the dynamic of elliptic, parabolic and loxodromic elements respectively Finally in 
 Section \ref{s:main} we write down some questions, concerning cyclic, that    where not considered in this article, 

\section{Preliminaries}  \label{s:recall}
\subsection{Projective Geometry}
The complex projective space $\mathbb {P}^n_{\mathbb {C}}$
is defined as:
$$ \mathbb {P}^{n}_{\mathbb {C}}=(\mathbb {C}^{n+1}\setminus \{0\})/\C^* \,,$$
where $\C^*$ acts by  the usual scalar multiplication.
  This is   a  compact connected  complex $n$-dimensional
manifold,  equipped with the Fubini-Study  metric $d_n$.

If $[\mbox{ }]:\mathbb {C}^{n+1}\setminus\{0\}\rightarrow
\mathbb {P}^{n}_{\mathbb {C}}$ is the quotient map, then a
non-empty set  $H\subset \mathbb {P}^n_{\mathbb {C}}$ is said to
be a projective subspace of dimension $k$  if there is a  $\mathbb {C}$-linear
subspace  $\widetilde H$ of dimension $k+1$ such that $[\widetilde
H\setminus \{0\}]=H$. The  projective subspaces of dimension
$(n-1)$ are called hyperplanes and the  complex projective subspaces of dimension 1
 are called   lines.  In this article  $e_1,\ldots, e_{n+1}$ will denote the standard basis for $\Bbb{C}^{n+1}$.\\

Given a set of points $P$   in $\mathbb{P}^{n}_{\mathbb{C}}$, we
define:
$$\langle\langle P \rangle\rangle =\bigcap\{l\subset \mathbb{P}^n_{\mathbb{C}}\mid l
\textrm{ is a projective subspace containing } P \}.$$ Clearly 
$\langle\langle P\rangle\rangle $ is a projective subspace of
$\mathbb{P}^{n}_{\mathbb{C}}$.  On the other hand the points in $P$ are  said to be in general position if for each subset $R\subset P$ with  $1\leq Card(R)\leq n+1$  we have that   $\langle\langle R \rangle\rangle$ has dimension $Card(R)-1$.

\subsection{The Grassmanian} 
Let $0\leq k< n$, then we define the Grassmanian   $Gr(k,n)$ as the space  of all  $k$-dimensional projective subspaces  of $\mathbb {P}^{n}_{\Bbb{C}}$ endowed with the Hausdorff topology. One has  that $Gr(k,n)$ is a  compact, connected complex manifold of dimension $k(n-k)$. 
A method to realize the Grassmannian  $Gr(k,n)$  as a sub variety of the projective space of the $(k+1)$-th exterior power of $ \C^{n+1}$, in symbols $P(\bigwedge^{k+1} \C^{n+1})$ is done by the so called Pl\"ucker embedding which is given by:
\[
\begin{array}{l}
\iota:Gr(k,n)\rightarrow P(\bigwedge^{k+1} \C^{n+1})\\
\iota(V)\mapsto  [v_1\wedge \cdots \wedge v_{k+1}]
\end{array}
\] 
where $\langle \langle v_1, \cdots,  v_{k+1}\rangle \rangle= V$, clearly this is a well defined $PSL(n+1,\Bbb{C})$-equivariant  embedding. Moreover, it is possible to check that the topology on $Gr(k,n)$ induced by the Fubini-Study metric on $P(\bigwedge^{k+1} \C^{n+1})$
agrees with the topology on $Gr(k,n)$. \\

\subsection{ Projective, Quasi-projective and Pseudo-projective transformations }
Consider the general linear group $ GL({n+1}, \C)$.
 It is clear  that every linear automorphism of $\C^{n+1}$ defines a
holomorphic automorphism of $\P^n$, and it is well-known that every holomorphic automorphism of $\P^n$
arises in this way. The group of projective
automorphisms of $\P^n$ is defined:
$$PSL(n+1, \mathbb {C}) \,:=\, GL({n+1}, \C)/\C^*$$
where $\C^* $ acts by the usual scalar multiplication. Then $PSL(n+1, \mathbb {C})$ is a Lie group whose
elements are called projective transformations.
We denote   by $[[\mbox{  }]]: GL(n+1, \mathbb
{C})\rightarrow PSL(n+1, \mathbb {C})$    the quotient map. Given     $ \gamma  \in PSL(n+1, \mathbb
{C})$  we  say that  $\widetilde \gamma  \in GL(n+1, \mathbb {C})$
is a  lift of $ \gamma  $ if $[[\widetilde  \gamma  ]]= \gamma  $.
 Notice that
$ PSL(n+1, \mathbb {C})$   takes
   projective
subspaces  into projective subspaces.\\

 Consider the space of linear transformations from $\C^{n+1}$ to $\C^{n+1}$ denoted by  $M(n+1,\C)$,  this is a linear complex space  of dimension  $(n+1)^2$ 
 where   $GL(n+1,\C)$ is an open dense set in $ M(n+1,\Bbb{C})$. Then
$PSL(n+1,\C)$ is an open dense set in  $QP(n+1,\C)= (M(n+1,\C)\setminus \{0 \})/\C^*$ called the space of pseudo-projective   maps , that is $QP(n+1,\C)$is a compactification, in the set theoretic sense, of $PSL(n+1,\C)$. 
Let us show how the previous compactification  can be used, set 
$\widetilde M:\mathbb {C}^{n+1}\rightarrow \mathbb {C}^{n+1}$ be a non-zero
 linear transformation   $Ker(\widetilde M)$ be its kernel and  $Ker([[\widetilde M]])$
 denote the respective projectivization,
 then  $\widetilde M$ induces a map 
  $[[\widetilde M]]:\mathbb {P}^{n}_\mathbb {C}\setminus Ker([[\widetilde M]]) \rightarrow
 \mathbb {P}^{n}_\mathbb {C}$   by:
 $$
[[\widetilde M]]([v])=[\widetilde M(v)]\,.
 $$
which is clearly well defined. The following  result provides a ``relation" between  convergence in $QP(n+1,\C)$ viewed as points in a projective space and the convergence viewed as functions.

\begin{proposition} [See \cite{CSN}]  \label{p:completes}
Let  $( \gamma_m)_{m\in \mathbb {N}}\subset PSL(n+1,\mathbb {C})$
be a sequence of distinct elements, then 
\begin{enumerate}
\item  There is a subsequence $( \tau_m)_{m\in \mathbb {N}}\subset ( \gamma_m)_{m\in \mathbb {N}}$ and  $\tau_0\in  M(n+1,\C)\setminus \{0\}$ such that $\tau_m\xymatrix{ \ar[r]_{m \rightarrow
\infty}&}  \tau_0  $ as points in $QP(n+1,\C)$.

\item If $(\tau_m)_{m\in \mathbb {N}}$ is the sequence  given by the previous part of this lemma, then 
$\tau_m  \xymatrix{ \ar[r]_{m \rightarrow
\infty}&}  \tau_0  $, as   functions,   uniformly on compact sets of 
 $\mathbb
{P}^n_\mathbb {C}\setminus Ker( \tau_0  )$.

\end{enumerate}
\end{proposition}

 As we will see below Proposition \ref{p:completes} enable us to describe in a simple way the equicontinuity region  of a given family of automorphism of $\P^n$. 
 
 
\subsection{Projective Unitary Groups} Let $0<k\leq l$ and consider  $\C^{k+l}$ equipped equipped with the
Hermitian form of signature $(k,l)$ given by:
$$
\prec u,v  \succ_{k,l}=-\sum_{j=1}^{k}u_j\overline{v}_j+\sum_{j=k+1}^{k+l}u_j\overline{v}_j \,,
$$
where $u=(u_1,\ldots,u_{k+l})$ and  $v=(v_1,\ldots,v_{k+l})$. A
vector $v$ is called negative, null and positive depending (in the
obvious way) on the value of  $\prec v,v \succ_{k,l}$; we denote the set of
negative, null or positive vectors by $N_-^{k,l},$ $N_0^{k,l}$ and $N_+^{k,l}$
respectively. We define $\H^{k,l}_\C$ as 
the image  of  $N_-^{k,l}$ in
$\P^n$ under the map $[\mbox{ }]$.\\

If we let  $U(k,l) \subset GL(n+1, \C)$  be the subgroup
consisting of the elements  that preserve the above Hermitian
form, then its projectivization $PU(k,l)=[[U(k,l)]]_{n+1}$  is a Lie  subgroup
of $ PSL(n+1, \mathbb {C})$ that we denote by $PU(k,l)$. In the case $k=1$,  the function $d_B:\H^{1}_{\C} \times \H^{l}_{\C} \rightarrow \Bbb{R}^+$ given by:
\[
d_B([v],[w])=
arccosh
\left(
\sqrt{
\frac
{\prec v,w\succ _{1,l}, \prec v,w\succ _{1,l}}
{\prec v,v\succ _{1,l}\prec w,w\succ _{1,l}}
} \,
\right )
\]
 is a metric in $\H^{l}_{\C}$ compatible with the   topology of $\H^{l}_{\C}$.  Where each element in $PU(1,l)$ acts on $\Bbb{H}^{l}_{\Bbb{C}} $ as an isometry. 
\begin{remark}
Let  $k>1$ and $l\geq k$. Consider the projective transformation induced by the matrix
 \[
M=
\left (
\begin{array}{lll}
I  & 0 & 0\\
0 & 1 & 1\\
0 & 0 & 1\\
\end{array}
\right )
\]
where $I$ is the identity matrix in $SL(k+l-2)$. Also consider the Hermitian quadric given by 
$$Q=-\sum_{j=1}^{k-1} \vert z_j\vert^2+ \sum_{j=k}^{k+l-2} \vert z_j\vert^2 +i(z_{k+l-1}\overline{z_{k+l}}-z_{k+l}\overline{z_{k+l-1}}),$$ 
clearly $Q$ has signature $(k,l)$,  $M$ preserves the $(k,l)$-ball $B$ induced by $Q$  and 
$$\langle\langle [e_1],\ldots, [e_{k-1}]\rangle \rangle \subset Fix([M])\cap B.$$  This example shows that for $k\geq 2$ not every discrete group in $PU(k,l)$ acts discontinuously on $\Bbb{H}^{k,l}_{\Bbb{C}}$,  
which, by the Arzela- Azcoli theorem, shows that we cannot find a metric in $\Bbb{H}^{k,l}_{\Bbb{C}}$ compatible with the topology  in such way that $PU(k,l)$ acts by isometries on  $\Bbb{H}^{k,l}_{\Bbb{C}}$. 
\end{remark}

\subsection{Kulkarni limit Set}
When we look at the action of a group acting on a general topological space, 
in  general there is not a natural notion of limit set, back in 1970 Kulkarni introduce a notion of limit set 
which work in very general setting,  in this subsection we present the  Kulkarni's limit set. 

\begin{definition}[see   \cite{kulkarni} ] \label{d:lim}
 Let $\Gamma\subset   PSL(n+1,\mathbb{C})$ be a subgroup. We  define
 
\begin{enumerate}
\item The set 
$\Lambda(\Gamma)$ as the closure of the set  of cluster points  of
$\Gamma z$  where  $z$ runs  over  $\mathbb{P}^n_{\mathbb{C}}$
\item  The set $L_2(\Gamma)$ as  the closure of cluster  points of $\Gamma
K$  where $K$ runs  over all  the compact sets in
$\mathbb{P}^n_{\mathbb{C}}\setminus \Lambda(\Gamma)$.
\item The  \textit{ Kulkarni's limit set } of $\Gamma$  as:  $$\Lambda_{Kul} (\Gamma) = \Lambda(\Gamma) \cup L_2(\Gamma).$$ 
\item The \textit{ Kulkarni's discontinuity region} of $\Gamma$  as:
$$\Omega_{Kul}(\Gamma) = \mathbb{P}^n_{\mathbb{C}}\setminus
\Lambda_{Kul}(\Gamma).$$
\end{enumerate}

\end{definition}

The Kulkarni's limit set has the following properties,  for a more  detailed  discussion on this topic in the 2 dimensional setting see \cite{CSN} and Corollary \ref{c:eqkul}  below.

\begin{proposition}[See \cite{CSN}] \label{p:pkg}
Let    $\Gamma$  be a complex  kleinian group. Then:

\begin{enumerate}

\item The sets\label{i:pk2}
$\Lambda_{Kul}(\Gamma),\,\Lambda(\Gamma),\,L_2(\Gamma)$
are  $\Gamma$-invariant closed sets. 

\item \label{i:pk3} The group $\Gamma$ acts properly 
  discontinuously on  $\Omega_{Kul}(\Gamma)$. 

\item \label{i:pk4} Let $\mathcal{C}\subset\mathbb{P}^n_{\mathbb{C}}$ be
a closed $\Gamma$-invariant set such that  for every compact set $K\subset
\mathbb{P}^n_{\mathbb{C}}\setminus \mathcal{C}$, the set of cluster points
of  $\Gamma K$ is contained in $\Lambda(\Gamma)\cap \mathcal{C}$, then
$\Lambda_{Kul}(\Gamma)\subset \mathcal{C}$.

\item The equicontinuity set of $\Gamma$ is contained in $\Omega_{Kul}(\Gamma)$.
\end{enumerate}
\end{proposition}

As we will in this article Kulkarni's limit set is a nice starting point  in order to study  the different notion of limit
 sets.
  
\subsection{The $\Lambda$ set for cyclic groups}
In the case of cyclic groups there is  available a full description of the set $\Lambda$, 
in order to present such result the following definition is necessary.

\begin{definition} Let $\gamma\in SL(n+1,\C)$ be a linear transformation transformation,  
$V_1,\ldots ,  V_k\subset \mathbb{C}^{n+1}$  linear subspaces; $k\in \mathbb{N}$;
$\gamma_i:V_i\rightarrow V_i$, $1\leq i\leq k$, be $\C$-linear transformations and
$r_1,\ldots,r_k\in \mathbb{R}$. The set  $(k,\{V_i\}_{i=1}^k,\{\gamma_i\}_{i=1}^k,\{r_i\}_{i=1}^k)$
 will be called a unitary decomposition for $\gamma$ if it is  verified that:
 
\begin{enumerate}
\item $\bigoplus_{j=1}^k V_j=\mathbb{C}^{n+1}$.
\item For each  $1\leq i\leq k$,  the eigenvalues of $\gamma_i$  are unitary complex numbers.
\item $0<r_1<r_2<,\ldots,<r_k$.
\item $\bigoplus_{j=1}^kr_j\gamma_j= \gamma$.
\end{enumerate}
\end{definition}

Now is trivial to provide a description  of $\Lambda$ for the cyclic case, see \cite{CSN}. 

\begin{theorem} \label{l:lila} 
 Let $\gamma\in  PSL(n+1, \mathbb{C})$ be a projective transformation with infinite order. If $\widetilde\gamma \in SL(n+1,\C)$ is a lift of $\gamma$ and $(k, \{V_j\}_{j=1}^k, \{\gamma_j\}_{j=1}^k,\{r_j\}_{j=1}^k)$  is a unitary decomposition for $\widetilde\gamma$, then:
$$ \Lambda(\langle \gamma \rangle )=
\bigcup_{j=1}^k \langle \langle [\left\{x\in V_j: x \textrm{ is an eigenvector of } \gamma_j\right \} ] \rangle \rangle .$$
\end{theorem}

 Thus the set $\Lambda(\langle\gamma \rangle )$ is a finite union of projective 
 subspaces, trivially $\Lambda(\langle\gamma \rangle )$ is a single proper 
 non-empty projective subspace whenever $\gamma$ is parabolic,
 it has at least two projective subspaces  if an only if $\gamma$ is a loxodromic
 and is either $\Bbb{P}^n_{\Bbb{C}}$ or $\emptyset $ when $\gamma$ is elliptic. Now we can prove one of our main results.

 \subsection*{Proof of Theorem \ref{t:l2}}
 \begin{proof}
 This result is a direct consequence  of Theorem \ref{l:lila}.
 \end{proof}

\section{ Describing the Equicontinuity set by means of Pseudo-projective maps }  \label{s:equi}

In this section we provide a description of the equicontinuity set using pseudo projective maps.
 Let us recall the definition of the equicontinuity set

\begin{definition}
The {\it equicontinuity region} for a family $G$ of
endomorphisms of  $\mathbb {P}^n_\mathbb {C}$, denoted $Eq(G)$, is
defined to be the set of points $z\in \mathbb{P}^n_\mathbb{C}$ for
which there is an open neighborhood $U$ of  $z$   such that $G
\vert_U$ is a normal family.  
\end{definition}

The following technical lemmas will be useful, compare with Corollary 4.3, Proposition 3.3 and Corollary 3.4
 in \cite{CSN}.

\begin{lemma} \label{l:pns}
Let $(\gamma_m)_{m\in \Bbb{N}}\subset  PSL(n+1,C)$ be a sequence of distinct elements and $\gamma,\theta\in QP(n,C)\setminus \PSL(n+1,\C)$, such that $\gamma_m \xymatrix{ \ar[r]_{m \rightarrow \infty}&} \gamma  $ and $\gamma_m^{-1} \xymatrix{ \ar[r]_{m \rightarrow \infty}&} \theta  $ in the sense of pseudo-projective transformations, then $Im(\gamma)
\subset  Ker(\theta)$, here $Im(\gamma) =\gamma(\P^n
\setminus  Ker(\gamma))$.
\end{lemma}
\begin{proof}
On the contrary, let us assume that  $Im(\gamma)\setminus  Ker(\theta)$ is a non-empty set, then $\gamma^{-1}(Im(\gamma)\setminus  Ker(\theta))$ is a non-empty open set, since  $Ker(\gamma)$ and 
$Im(\theta)$ are proper projective subspaces we can choose $\tilde p\in \gamma^{-1}(Im(\gamma)\setminus  Ker(\theta))\setminus (Ker(\gamma)\cup Im(\theta))$, then $\tilde p = \gamma^{-1}_m(\gamma_m(\tilde p))
\xymatrix{ \ar[r]_{m \rightarrow \infty}&} \theta(\gamma(\tilde p))$, which is a contradiction.
\end{proof}

\begin{lemma} \label{l:eqsuc}
Let $(\gamma_m)_{m\in\Bbb{N} }\subset PSL(n+1,\C)$ be a sequence of distinct 
elements, $\gamma\in QP(n+1,\C)\setminus PSL(n+1,\C)$ and $\mathfrak{L}$ a 
projective subspace of $\P^n$, such that $(\gamma_m)$ converges to $\gamma$ 
in the sense of pseudo-projective transformations, $\mathfrak{L}\cap 
Ker(\gamma)$ is a single point $p$ and $dim (\mathfrak{L})\geq dim 
(Im(\gamma))+1$, then there is a subsequence $(\tau_m)_{m\in \Bbb{N}}\subset 
(\gamma_m)_{m\in \Bbb{N}}$ and a projective subspace $\mathcal{L}$ such that 
$dim(\mathcal{L})=dim(\mathfrak{L})$, $Im(\gamma)\cap \mathcal{L}\neq 
\emptyset$ and for every point $y\in \mathcal {L}$ there is a sequence 
$(x_m)_{m\in \Bbb{N}}\subset \mathfrak{L}$ satisfying $\tau_m(x_m) 
\xymatrix{ \ar[r]_{m \rightarrow \infty}&} y $ and $x_m \xymatrix{ 
\ar[r]_{m \rightarrow \infty}&} p $.
\end{lemma}
\begin{proof}
Let $k=dim (\mathfrak{L})$. Since $Gr(k,n)$ is compact, there is a sub
sequence $(\tau_m)_{m\in \N}\subset (\gamma_m)_{m\in \N}$ and $\mathcal{L}\in 
Gr(k,n)$ such that $\tau_m(\mathfrak{L}) \xymatrix{ \ar[r]_{m \rightarrow 
\infty}&} \mathcal{L} $. Let $y\in \mathcal {L}\setminus Im(\gamma)$, then 
there is convergent sequence $(x_m)_{m\in \Bbb{N}}\subset \mathfrak{L}$ such 
that $\tau(x_m)\xymatrix{ \ar[r]_{m \rightarrow \infty}&} y$. This implies 
that the limit point of $(x_m)_{m\in \Bbb{N}} $ lies in $Ker(\gamma)$, thence such point is $p$. Finally, since $\tau_m \xymatrix{ \ar[r]_{m \rightarrow \infty}&} \gamma$ and $\mathfrak{L}\cap Ker(\gamma)=\{p\}$, we conclude that $Im(\gamma)\cap \mathcal{L}\neq \emptyset$ 
\end{proof}

\begin{lemma} \label{l:eqsuc2}
Let $(\gamma_m)_{m\in\Bbb{N} }\subset PSL(n+1,\C)$ be a sequence of distinct 
elements, $\gamma\in QP(n+1,\C)\setminus PSL(n+1,\C)$, such that $(\gamma_m)$ converges to $\gamma$ 
in the sense of pseudo-projective transformations, then:  $$Eq(\{\gamma_m\vert m\in \Bbb{N}\})=\P^n\setminus Ker(\gamma).$$
\end{lemma}
\begin{proof}
On the contrary let us assume  that there is $$p\in Ker(\gamma)\cap Eq(\{\gamma_m\vert m\in \Bbb{N}\}).$$ 
Let $\mathcal{Y}\subset \P^n$ be a projective subspace disjoint from  $ Ker(\gamma)$  and satisfying  $dim(\mathcal{Y})=dim (Im(\gamma))$. Define $\mathcal{L}=\langle\langle \mathcal{Y}, x\rangle\rangle$, 
then  by   Lemma \ref{l:eqsuc} there is a subsequence $(\tau_m)_{m\in\Bbb{N}}\subset (\gamma_m)_{m\in\Bbb{N}} $ and $\mathfrak{L}\subset \P^n$ a projective subspace such that  $\mathfrak{L}\cap Im(\gamma)\neq \emptyset$,  $dim(\mathcal{L})=dim(\mathfrak{L})$ and for every point $y\in \mathfrak{L}$ there is a sequence $(y_m)_{m\in\Bbb{N}}\subset \mathcal{L}$ such that $y_m\xymatrix{ \ar[r]_{m \rightarrow \infty}&}  p$ and $\gamma_m(y_m)\xymatrix{ \ar[r]_{m \rightarrow \infty}&}  y$. Let $z,w\in \mathfrak{L}\setminus Im(\gamma)$ be distinct points, then by the previous  argument there are  sequences $(z_m)_{m\in \Bbb{N}},(w_m)_{m\in \Bbb{N}}\subset\mathcal{L}$ such that $w_m,z_m\xymatrix{ \ar[r]_{m \rightarrow \infty}&}  p$ and $\gamma_m(z_m)\xymatrix{ \ar[r]_{m \rightarrow \infty}&}  z$ and $\gamma_m(w_m)\xymatrix{ \ar[r]_{m \rightarrow \infty}&}  w$. Finally, from the definition of equicontinuity,  there is a subsequence $(\kappa_m)_{m\in \Bbb{N}}\subset (\tau_m)_{m\in \Bbb{N}}$  and $ \kappa:Eq(\{\gamma_m\vert m\in \Bbb{N}\})\rightarrow \P^n$ holomorphic such that $\kappa_m\xymatrix{ \ar[r]_{m \rightarrow \infty}&}  \kappa$  in the compact open-topology. Since $p\in Eq(\{\gamma_m\vert m\in \Bbb{N}\})=\P^n\setminus Ker(\gamma)$, our previous arguments  $z=\kappa(p)=w$, which is a contradiction.
\end{proof}

Now let us describe the equicontinuity set in terms of pseudo-projective maps, compare with
 Lemma 5.1 in \cite{CSN}.

\begin{proposition}  \label{p:eq} Let
$\Gamma \subset PSL(n+1,\mathbb {C})$ be a group and define

\begin{scriptsize}
$$
Lim(\Gamma)=
\{\gamma\in QP(n+1,\C): \textrm{ there is }( \gamma_m)_{m\in \Bbb{N}}\subset \Gamma, \textrm{ of distinct elements with }
\gamma_m  \xymatrix{ \ar[r]_{m \rightarrow  \infty}&} \gamma 
\},
 $$
\end{scriptsize}
then $$Eq(\Gamma)=\P^n\setminus \overline{\bigcup_{\gamma\in Lim(\Gamma)} Ker(\gamma)}.$$
 \end{proposition}
\begin{proof}   From Proposition \ref{p:completes}  we deduce
 $\P^n\setminus \overline{\bigcup_{\gamma\in Lim(\Gamma)} Ker(\gamma)}\subset Eq(\Gamma)$. Now 
 the  result follows from Lemma \ref{l:eqsuc2}.
\end{proof}

Observe that the previous result remains valid for an arbitrary family of projective transformations.   As a corollary we get the following result, compare with Corollary 5.3 in \cite{CNS}.

\begin{corollary} \label{c:eqkul}
Let $\Gamma\subset PSL(n+1,\C)$ be a discrete  group, then $\Gamma$ acts properly discontinuously on $Eq(\Gamma)$. Moreover $Eq(\Gamma)\subset \Omega_{Kul}(\Gamma)$.
\end{corollary}
\begin{proof}
Assume on the contrary that $\Gamma$ does not act discontinuously on $Eq(\Gamma)$. Then
there is a compact set $K$ and a sequence of distinct elements $(\gamma_m)_{m\Bbb{N}} \subset  \Gamma$, such that
$\gamma(K)\cap K\neq \emptyset$.  By Proposition \ref{p:completes}, there is a subsequence of $(\tau_m)_{m\in \Bbb{N}}\subset (\gamma_m)_{m\in \Bbb{N}}$, and $\tau \in  QP(n + 1,\C)\setminus PSL(n+1,\C)$, such that $(\tau_m)_{m\in \Bbb{N}}$ converges to $\tau$ in the sense
of pseudo-projective transformations.  Moreover, by Lemma \ref{l:pns} we know that $ Im(\tau )$ is  a projective subspace contained in $\P^n\setminus Eq(\Gamma)$.
 Therefore there is a neighborhood $U$ of $Im(\tau)$
disjoint from $K$ and a natural number $n_0$ such that $\tau_m(K)\subset U$ for all $m>n_0$. This
implies $\tau_m(K)\cap K =\emptyset $, which is a contradiction. Therefore $\Gamma$ acts discontinuously
on $Eq(\Gamma)$. 

Finally, from the previous argument we deduce  that for every compact set $K \subset Eq(\Gamma)$
the cluster points of $\Gamma K$ lie in $L_1 (\Gamma)$, thus  by  Proposition \ref{p:pkg}  we conclude that  $Eq(\Gamma)\subset \Omega_{Kul}(\Gamma)$.
\end{proof}

 Now as a consequence of the results of this sections we are able to compute  the equicontinuity set for cyclic groups. The following notation will be helpful.

\begin{definition}
	Let $V$ be a $\C$-vector space, $T:V\rightarrow V$ be a $\C$-linear map whose eigenvectors are unitary complex values.  We will say $( \{V_j\}_{j=1}^k,\{\beta_j\}_{j=1}^k, \{T_j\}_{j=1}^k)$ is a Block decomposition for $T$ if:
	
	\begin{enumerate}
		\item We have $V_1,\ldots,  V_k\subset V$ are
	linear subspaces satisfying $\bigoplus_{j=1}^k V_j=V$.
	\item  For $j>2 $ we have $\beta_j=\{v_{j1},\ldots, v_{jn_j} \}$ is an ordered  base of $V_j$. 
	
		\item For     $1\leq j\leq k$  there is $T_j:V_j\rightarrow V_j$ a $\C$-linear maps
		such that   $T_1$ is diagonalizable and the matrix associated to $T_j$ with respect   $\beta_j$ is  a Jordan block in other case.
		\item $\bigoplus_{j=1}^kT_j=T$. 
	\end{enumerate}
\end{definition}
Given a unitary linear map we can define the following in terms of a block decomposition:
$$
H(T)=max(\{dim V_j:j\in \{2, \ldots, k \} \});
$$
$$
\Xi_j(T)
=\left \{
\begin{matrix}
\langle\langle \beta_j\setminus \{v_{jn_j}\}\rangle \rangle & \textrm{ if } j\neq 1 \textrm{ and } 
dim V_j=H(T)
V_j & \textrm{ in other case }\\
\end{matrix}
\right .
$$

$$
\Xi(T)=Span \left( \bigcup_{j=1}^k  \Xi_j(T) \right).
$$

Clearly $\Xi(T)$ does not depend on the choice of the block decomposition.

\begin{theorem} \label{t:equi}
	Let $\gamma\in  PSL(n+1, \mathbb{C})$ be a projective transformation with infinite order. If $\widetilde\gamma \in SL(n+1,\C)$ is a lift of $\gamma$,  
	$(k, \{V_j\}_{1}^k, \{\gamma_j\}_{1}^k,\{r_j\}_{1}^k)$ is a unitary decomposition for $\widetilde\gamma$, and then
	\[
	\P^n\setminus Eq(\langle\gamma  \rangle )= 
	\left [
	Span \left (  
	\bigcup_{j>1} V_j
	\cup
	\Xi(\gamma_1)
	\right ) \setminus \{0\}\right ]\cup
	\left [
	Span\left (  
	\bigcup_{j<k} V_j
	\cup \Xi(\gamma_k)
	\right ) \setminus \{0\}\right ].
	\]
\end{theorem}

\begin{proof} Let $T\in Lim(\langle \gamma\rangle)$, then  there is 
	a  sequence  $(n_m )\subset \Z$  such that $\gamma^{n_m}  \xymatrix{ \ar[r]_{m \rightarrow  \infty}&} T $.  After taking a subsequence, if it is necessary, we can assume that either $(n_m)$ is negative or $(n_m)$ is positive. Without loss of generality  let us assume that $n_m>0$. Let 
	$(\{U_j\}_{1}^{r}, \{\beta_{j}\}_{1}^{r}, \{\rho_j\}_{j=1}^r)$  be a block decomposition for $\gamma_k$, define 
	$A_1=\{j\in \{2,\ldots,r \}: dim U_j<H(\gamma_k) \}\cup \{1 \}$ and $A_2=\{1,\ldots,r \}\setminus A_1$,  then a straightforward calculation shows:

	\begin{scriptsize}
		\begin{equation}\label{e:eq0}
			\frac{
					\tilde{\gamma}^{n_m}
				}
					{	\left ( \begin{array}{c}n_m\\ H(\gamma_k)-1\end{array}\right)r_k^{n_m} }
=
		\sum_{j<k}
\frac{r_j^{n_m} \gamma_j^{n_m}}{
		\left (
\begin{array}{c}n_m\\ H(\gamma_k)-1\end{array}
\right)r_k^{n_m}
}
		+
		\sum_{j\in A_1}
		\frac{\rho _j^{n_m}}{		\left ( \begin{array}{c}n_m\\ H(\gamma_k)-1\end{array}\right)}
		+\sum_{j\in A_2}
	\frac{\rho_j^{n_m} }{	\left ( \begin{array}{c}n_m\\ H(\gamma_k)-1\end{array}\right)}
		\end{equation}
	\end{scriptsize}
	By  Perron-Frobenius  theorem (see \cite{CSN}) we conclude:
	
	\begin{equation}\label{e:eq1}
	\sum_{j<k}
	\left (
	\begin{array}{l}n_m\\ H(\gamma_k)\end{array}
	\right)^{-1}r_k^{-n_m}r_j^{n_m} \gamma_j^{n_m}   \xymatrix{ \ar[r]_{m \rightarrow  \infty}&} 0 \textrm{ point-wise}
	\end{equation}
	\begin{equation}\label{e:eq2}
	\sum_{j\in A_1}
	\left ( \begin{array}{l}n_m\\ H(\gamma_k)\end{array}\right)^{-1}\rho _j^{n_m}  \xymatrix{ \ar[r]_{m \rightarrow  \infty}&} 0 \textrm{ point-wise}
	\end{equation}
	Now  let $(k_m)\subset (n_m)$ be a subsequence such that  for each $j\in A_2$ it  is verified that 
	$\rho_j^{k_m}(u_{1j}) \xymatrix{ \ar[r]_{m \rightarrow  \infty}&} \vartheta_ju_{1j}$, where $\vartheta_j \in \Bbb{S}^1 $ for all $j\in A_2$. Now let  $j\in A_2$ and $\beta_j=\{u_{ji}\}_{i=1}^{H(\gamma_k)}$  define the linear transformation  
	$S_{j}=U_j \rightarrow U_j$ induced by 
	$$
	S_j(u_{ji})=
	\left \{ 
	\begin{array}{ll}
	0                     &  \textrm{ if } i< H(\gamma_k)  \\
	\vartheta_j  u_{j1}  & \textrm{ otherwise }
	\end{array}
	\right.
	$$
	Clearly 
	\begin{equation}\label{e:eq3}
	\sum_{j\in A_2}
	\left ( \begin{array}{l}n_m\\ H(\gamma_k)\end{array}\right)^{-1}\rho_j^{n_m} 
	\xymatrix{ \ar[r]_{m \rightarrow  \infty}&} \sum_{j\in A_2} S_j \textrm{ point-wise}
	\end{equation}
	
	Finally,  for each $j\in A_1$ let $S_j:U_j \rightarrow U_j$ given by  $S_j=0$ and for  $j\in  \{1,\ldots, k\}$ a define    $T_j=V_j \rightarrow V_j$ by 
	$$
	T_j=
	\left \{ 
	\begin{array}{ll}
	0                     &  \textrm{ if }  j< k  \\
	\sum_{l=1}^r S_{l}  & \textrm{ if } j=k.
	\end{array}
	\right.
	$$
	
	Clearly from  Equations \ref{e:eq0}, \ref{e:eq1}, \ref{e:eq2} and \ref{e:eq3} we conclude
	$
	\gamma ^{k_m}
	\xymatrix{ \ar[r]_{m \rightarrow  \infty}&}
	\left[\left[ \sum  T_j\right ] \right ].
	$
	Therefore $\left [ \left[ \sum  T_j\right ] \right ]= T$. Which concludes the proof.
\end{proof}

Now we can prove another of our main results.

\subsection*{Proof of Theorem \ref{t:e2}}
\begin{proof}
This result   is straightforward in view of  Theorem \ref{t:equi}.
\end{proof}

It is possible to construct examples (see \cite{CNS}) of non virtually cyclic groups  whose  equicontinuity 
set is a proper set of the Kulkarni's discontinuity set.

\section{Elliptic Transformations } \label{s:eli}

In this section we are going to study the dynamic of elliptic transformation, in particular we are going to show 
that   any elliptic transformation can be conjugated to an element in $PU(k,l)$. Let us start with some technical 
results.

\begin{lemma} \label{l:int}
Let $l \geq k>0$ and  $v,w\in N^{k,l}_0$ be linearly independent elements, then the quadratic form restricted to $\langle\langle v,w\rangle\rangle $ is either identically  $0$ or has signature $(1,1)$. 
\end{lemma}
\begin{proof}
By the theory of quadratic forms we know that $\prec, \succ_{k,l}$ restricted  to $\langle\langle v,w\rangle\rangle $ is  either $0$ or is equivalent to one of the following quadratic forms:
\[
\begin{array}{l}
\vert z_1\vert^2;\,
\vert z_1\vert^2+\vert z_2\vert^2 ;\,
-\vert z_1\vert^2 ;\,
-\vert z_1\vert^2 -\vert z_2\vert ^2 ;\,
\vert z_1\vert ^2 -\vert z_2\vert ^2;
\end{array}
\]
Since there are two linearly independent null points, we conclude that $\prec, \succ_{k,l}$ restricted  to $\langle\langle v,w\rangle\rangle $ is  either $0$ or has signature $(1,1)$.
\end{proof}

\begin{corollary} \label{c:pronn}
Let $l \geq k>0$ and $\gamma\in PU(k,l)$ be an elliptic element, then  $ \gamma$ has fixed point in $\P^{k+l-1}\setminus \partial \Bbb{H}^{k,l}_\C$.
\end{corollary}
\begin{proof}
On the contrary, let us assume that each   fixed point of $\gamma$ lies on $\partial \H^{k,l}_\C$. Then if $\widetilde\gamma\in U(k,l)$ is a lift of $\gamma$, then we can find  a basis   $\beta=\{v_1,\ldots, v_{k+l}\}$ of $\C^{k+l}$ made up of  eigenvectors of $\widetilde \gamma$. \\

Claim.- Let $v,w\in \beta$ be distinct points, then     $\prec, \succ_{k,l}$ restricted  to
 $\langle\langle v,w\rangle\rangle $ cannot has signature $(1,1)$.  If this is not the case, we get that 
    $\gamma$ restricted to $\langle\langle [v],[w]\rangle\rangle $  is conjugated to an element in  $PU(1,1)$
      with  two fixed points  in $\partial \H^{1,1}_{\Bbb{C}} $, that is   $[[\gamma]]$ restricted to $\langle\langle [v],[w]\rangle\rangle$  is the identity, therefore $\gamma$ has a fixed point $\H^{k,l}_{\C}$ which is a contradiction.\\

Finally, by lemma 
\ref{l:int}, it yields that $\prec v, w \succ_{k,l}=0$. That is $\prec, \succ_{k,l}$ is identically $0$, which is a contradiction.
\end{proof}

\begin{lemma} \label{l:df} 

Let $l\geq k>0 $ and $\gamma\in PU(k,l)$  be an elliptic element, then  $ \gamma$ has fixed point in 
$\P^{k+l-1}\setminus \overline{ \H^{k,l}_\C}$.
\end{lemma}
\begin{proof}  On the contrary  let us assume that $Fix(\gamma)\subset \overline{\H^{k,l}_\C}$. Let $\widetilde \gamma\in U(k,l)$ be a lift of $\gamma$,
 by Corollary  \ref{c:pronn} there is  an   eigenvector  $v_1$   of $\widetilde\gamma$ such that
  $\prec v_1, v_1\succ_{k,l}<0$. Clearly   the Hermitian form $\prec, \succ_{k,l}$ restricted to  $\{v_1\}^\bot$ 
  has signature $(k-1,l)$. Since $\{v_1\}^\bot$ is a $\widetilde\gamma$-invariant subspace not containing 
  $v_1$,   Corollary \ref{c:pronn} ensures  that there is an eigenvector of $\gamma$ in  $\{v_1\}^\bot$  
  namely $v_2$ such that  $ v_2\in N^{k,l}_-$. Trivially  $\{v_1\}^\bot\cap \{v_2\}^\bot $ is a 
  $\widetilde \gamma$-invariant subspace, whose intersection with $Span(v_1,v_2)$ is trivial and where  
  the restriction of the  Hermitian form $\prec, \succ_{k,l}$  has signature $(k-2,l)$. 
  By an exhaustive process we can ensure the existence of a $\widetilde \gamma-$invariant subspace $W$ 
  where  the restriction of the  Hermitian form $\prec, \succ_{k,l}$  has signature $(0,l)$. Therefore there is  
  there is an eigenvector of $\gamma$ in  $W$  namely $w$ such that $\prec w, w\succ_{k,l}>0$, which is a 
  contradiction.
\end{proof}

From the proof of Lemma  \ref{l:df} we get the following corollary.

\begin{corollary} \label{l:disf}
Let $l \geq k>0$ and  $\gamma\in PU(k,l)$ be an elliptic element,  then there are   sets 
  $V_-\subset Fix(\gamma)\cap \H^{k,l}_\C$ and  
  $V_+\subset Fix(\gamma)\cap (\P^{k+l-1}\setminus \overline{\H^{k,l}_\C})$ such that 
     $Card(V_-)=k$, $Card(V_+)=l$ and $\langle \langle V_+\cup V_- \rangle \rangle =\P^{k+l-1}$. 
\end{corollary}

\begin{proposition} \label{p:eli}
Let $\gamma\in PU(k,l)$ be an  element. Then  $\gamma$ is elliptic if and only if   $\gamma$ has $k+l$ fixed points   in general position, where  $k$ of this points lie in $\H^{k,l}_\C$ and the other  $l$ are  in $\P^{k+l-1}\setminus\overline{\H^{k,l}_\C}$.
\end{proposition} 
\begin{proof} By   corollary \ref{l:disf} we only need to prove the converse of the proposition.
 Let $\gamma\in PU(k,l)$ be an elliptic element, and $\widetilde \gamma\in U(k,l)$ be a lift of $\gamma$, then by hypothesis we can find a basis $\beta=\{v_1,\ldots,v_{k+l}\}$ made up of eigenvectors where the first $k$ are in $N^{k,l}_-$ and the rest are in $N^{k,l}_+$. Let $M\in GL(k+l,\C)$ be  the change basis matrix which take the base $\beta$ into the standard basis of $\C^{k+l}$, then $M^{-1}\widetilde \gamma M$ is a diagonal matrix, let us said
\[
M^{-1}\widetilde \gamma M=
\left (
\begin{array}{lll}
\alpha_1\\
              & \ddots\\
              &            &\alpha_{k+l}
\end{array}
\right ).
\]

Let  $r=max\{\vert \alpha_1\vert,\ldots ,\vert \alpha_{k+l}\vert \}$, 
$$
\mathcal{I}^+=\{i:1,\ldots,k+l\vert (\alpha_i^m/r^{m})_{m\in \N} \textrm{ is bounded away from 0} \},
$$
 $(l_m)\subset \N$ such that $\alpha_i^mr^{-m} \xymatrix{ \ar[r]_{m \rightarrow  \infty}&} \check\alpha_i$ for every $i:1,\ldots, k+l$ and    $F:\C^{k+l}\rightarrow \C^{k+l} $ the linear map induced by   
$F(e_i)=\check\alpha_i e_i$. Then $\gamma^{l_m}\xymatrix{ \ar[r]_{m \rightarrow  \infty}&}[M][F] [M]^{-1}=\tau$ as points in $QP(k+l,\C)$. If $\gamma$ is not elliptic then 
$$
 \P^{k+l-1}\neq Im(\tau)=[[M]](<<\{[e_i]:i\in \mathcal{I}\}>>),$$ therefore  
 $Im(\tau)\cap Span(\beta) \neq \emptyset$. Since $\H^{k,l}_{\Bbb{C}}$  and $\P^{k+l-1}\setminus \overline {\H^{k,l}_{\Bbb{C}}}$ are  open and 
 $\gamma$-invariant we deduce that $Im(\tau)\subset \partial  \H^{k,l}_{\Bbb{C}} $,  which is a contradiction. 
\end{proof}

\begin{proposition} \label{p:fixp}
Let $l\geq k>1$ and  $\gamma\in PU(k,l)$, then $\gamma$ is elliptic if and only if there is a $\gamma$-invariant 
projective subspace $\mathcal{V}\subset \H^{k,l}_{\Bbb{C}}$ of dimension $k-1$.
\end{proposition}
\begin{proof} 
Now if $\gamma$ has an invariant space $\mathcal{V}\subset \H^{k,l}_\C$ of dimension $k-1$, then $\mathcal{V}^\bot$ is a $\gamma$-invariant space of dimension $l-1$ disjoint from $\mathcal{V}$. Therefore we can pick up $v_1,\ldots, v_k\in \mathcal{V}$ linearly independent fixed points and $v_{k+1},\ldots, v_{k+l}\in \mathcal{V}^\bot$  linearly independent fixed points, trivially $\P^{k+l-1}=\langle \langle v_1,\ldots, v_{k+l}\rangle\rangle$ which concludes the proof.
\end{proof}

In view of proposition \ref{p:fixp}, one can naively think that 
 loxodromic transformations in $PU(k,l)$ (respectively  parabolic) elements can be classified by saying simply 
that this elements correspond  to those elements  having    exactly two  (resp. one )
invariant subspaces in $\partial \Bbb{H}^{k,l}_{\Bbb{C}}$, the following 
examples shows  us this can not be done.

\begin{example}
Consider the following    matrices
\[
M=
\begin{pmatrix}
  0 & 0  & 0    & 1\\
  0 & 0  & -1 &1\\
  0 & -1 & 0    &0\\ 
 1&1&0&0\\  
\end{pmatrix},\,
B=
\begin{pmatrix}
  1 & 1  & 0    & 0\\
  0 & 1  & 0 &0\\
  0 & 0 & 1    &1\\ 
 0&0&0&1\\  
\end{pmatrix},\,
C=
\begin{pmatrix}
  2 & 2  & 0    & 0\\
  0 & 2  & 0 &0\\
  0 & 0 & 2^{-1}    &2^{-1}\\ 
 0&0&0&2^{-1}\\  
\end{pmatrix}.
\]
A straightforward calculation shows that $M$ is a Hermitian matrix  with  signature $(2,2)$,  $B$ induces a 
parabolic transformation,
$C$ induces a loxodromic transformation, we have the relations  $C^*M C=M= B^*M B$ and
$vCv^*=0$  whenever 
$v\in Span( e_1,e_2)\cup Span( e_3,e_2)\cup Span( e_3,e_4)$
\end{example}  

Now let us provide geometrical characterization of elliptic elements in $PU(k,l)$ in terms of $(k,l)$-spheres.   First recall that   elliptic elements in $ \PU(1,2)$  preserves a foliation in $\H^2$ made up      ``concentric  $3$-spheres" and it is know that such property characterizes elliptic elements in $PU(1,2)$, see \cite{CSN}, in this section we will extend  this result to the general setting, let us start with a definition 

\begin{definition} \label{3-spheres}
 The   family  
$$\mathcal{F}_{k,l}=
\left 
\{
T_r^{k,l} =
\left \{ 
[z_1 ,\ldots , z_{n+1} ] \in \P^n  \, : \,
\sum_{j=1}^k |z_j|^2=r\sum_{j=k+1}^{k+l} |z_{j}|^2 
\right
\}\,, \,r>0 
\right \},
$$ is the concentric foliation of  $(k,l)$-spheres. 
\end{definition}

The family 
$\mathcal{F}_{k,l}$
  provides  a
 foliation of $\P^ n
\setminus(\langle\langle e_1,\ldots, e_k\rangle \rangle  \cup \langle\langle e_{k+1},\ldots, e_{k+l}\rangle \rangle)$ and  $\partial \H^{k,l}$ is a leave.
In order to characterize the elliptic elements in terms of the foliation  $\mathcal{F}_{k,l}$  
the following standard result from linear algebra will be necessary

\begin{theorem}[Witt's extension theorem, see \cite{neretin} ]
Let $V_1, V_2 \subset  \C^{k+l}$  be 
subspaces where the restriction of the hermitian form $\prec, \succ_{k,l}$ is non degenerate. Then any isometry $f : V_1 \rightarrow V_2$ extends to an isometry $F$ of $\C^{k+l}$.

\end{theorem} 

\begin{proposition} \label{elipticdef}
A transformation $\gamma \in PU(k,l)$ is 
elliptic if and only if  there is $h\in PU(k,l)$ such that $\gamma$ leaves invariant each leave of the foliation $h(\mathcal{F}_{k,l})$.
\end{proposition}
\begin{proof}
Let  $\gamma \in PU(k,l)$ be an elliptic element and $\widetilde \gamma$ a lift of $\gamma$. Then by Proposition \ref{p:fixp} there is a projective subspace $V\subset N^{k,l}_-$ of dimension $k$ and invariant under $\widetilde\gamma$. Pick up vectors $v_1, \ldots,v_k$ such that $\prec v_i,v_j \succ_{k,l}$ is $0$ if $i\neq j$ and $-1$ in other case, then considering the linear extension $F$ of the  function  $f(e_i)=v_i$  we get an isometry from $U=Span(e_1,\ldots,e_k)$ to $W=Span(v_1,\ldots,e_v)$  and the restriction  of   of the hermitian form $\prec, \succ_{k,l}$  to $U$ and $W$ is negative definite By Witt's extension theorem $F$ extends to an element $\check F\in PU(k,l)$, clearly 
$$
\check F^{-1}\widetilde \gamma \check F
=\left (
\begin{array}{ll}
M\\
& N\\
\end{array}
\right ),
$$
where $M\in U(k)$ and $N\in U(l)$. Let $\widetilde M\in U(k)$ and $\widetilde N\in U(l)$ be such that
$\widetilde M^{-1}M \widetilde M$ and $\widetilde N^{-1}N \widetilde N$ are diagonal. In order to conclude this part of the proof we must define
$$
h=
\check F
\left (
\begin{array}{ll}
\widetilde M\\
& \widetilde N\\
\end{array}
\right ).
$$

Now assume that given $\gamma,h\in PU(k,l)$  we know that $\gamma$ preserves each leave of the foliation $h(\mathcal{F}^{k,l})$, then $h^{-1}\gamma h\in PU(k,l)$ and leaves invariant each leave of the foliation  $\mathcal{F}^{k,l}$, therefore $h^{-1}\gamma h$ leaves $\langle\langle e_1,\ldots,e_k\rangle\rangle\subset \H^{k,l}_\C$ invariant. By proposition \ref{p:fixp} we conclude that $\gamma$ is elliptic.  
\end{proof}

Now the following result is easily proven.

\begin{corollary}  \label{neweliffconjofoldeli} 
Let $l\geq k>0$, then  the following fact are equivalent:
\begin{enumerate} 
\item The element  $ \gamma \in  {\rm PSL} (n+1, \C)$ is elliptic, 
\item There is  $h\in PSL(n+1,\C)$ such that $\gamma$ leaves invariant $h(\partial \H^{k,l}_\C)$ and $h(\langle \langle e_1,\ldots, e_k\rangle\rangle)$,
\item There is  $h\in PSL(n+1,\C)$ such that $\gamma$ leaves invariant  each leave of the foliation  $h(\mathcal{F}^{k,l})$.
\end{enumerate}
\end{corollary}

Now we have.

\subsection*{Proof of Theorem \ref{t:e1}}
\begin{proof}
This result is a direct consequence of Corollary 	\ref{neweliffconjofoldeli}.  
\end{proof}
To conclude this subsection let us  describe the Kulkarni's limit set and the equicontinuity set for the cyclic  group generated by an elliptic element.

\begin{theorem}\label{t:ke}
Let $\gamma\in PSL(n+1,\C)$ be an elliptic element, then 
\begin{enumerate}
\item The limit set $\Lambda_{Kul}(\langle \gamma \rangle )$ is either empty or the whole space $\P^n$ according $\gamma$ having  finite or infinite order.
\item The equicontinuity set  $E(\langle \gamma \rangle )=\P^n$.
\end{enumerate}
\end{theorem}

Trivially $\Omega_{Kul}(\langle\gamma\rangle)$ is the largest open set on which $\langle\gamma\rangle$ 
acts properly discontinuously.

\section{Parabolic Transformations}  \label{s:par}
In this section we describe parabolic elements and its  relationship  with $PU(k,l)$. Our interest with other  semi-definite sesquilinear forms arises from the following: recall that in the one and two dimensional case is not hard to see that  parabolic elements are simply those coming from $PU(1,n)$, where $n=1,2$, but as the following example shows, in higher dimensions  there are parabolic elements which are not conjugated to an element in $PU(1,n)$.

\begin{example}
	Let $n\geq 4$ and $\gamma$ be  the projective transformation in $PSL(n+1,\C)$ induced by the matrix
	\[
	\widetilde \gamma=
	\left (
	\begin{array}{lllll}
	1 & 1\\
	0 & 1 \\
	&  &1 & 1\\
	&  &0 & 1 \\
	&  &  &  & I_{n-4}\\
	\end{array}
	\right )
	\]
	where $I_{n-4}$ is the identity matrix if $n>4$ and nothing in other case, then $\gamma$ is a parabolic element which cannot be conjugated to an element in $PU(1,n)$.
\end{example}
\begin{proof}
	A straightforward calculation shows that 
	\[
	\gamma^{\pm m}=
	\left [
	\left [
	\begin{array}{lllll}
	1 & \pm m\\
	0 & 1 \\
	&  &1 & \pm m\\
	&  &0 & 1 \\
	&  &  &  & I_{n-4}\\
	\end{array}
	\right ]
	\right ]
	\xymatrix{ \ar[r]_{m \rightarrow  \infty}&}
	\left [
	\left [
	\begin{array}{lllll}
	0 & 1\\
	0 & 0 \\
	&  &0 & 1\\
	&  &0 & 0 \\
	&  &  &  & 0_{n-4}\\
	\end{array}
	\right ] 
	\right ]
	\]
	in consequence $\P^n\setminus Eq(\langle \gamma\rangle )$ is not an hyperplane, which is not  possible for elements in $PU(1,n)$, see \cite{CNS}. Therefore, $\gamma$ is not conjugate to an element in $PU(1,n)$. 
\end{proof}

So in order to describe parabolic elements we will need to understand the relationship between
  Jordan blocks and Hermitian products.  Let us consider the following functions:

  \begin{definition}
  	Given  $n>2$ and $k,j:1,\ldots n$, let us 
  	define $$\ll i,j\gg^{(+)},\ll i,j\gg^{(-)}:\C^n\rightarrow \Bbb{R}$$ by 
  	\[
  	\begin{matrix}
  	\ll k,j\gg^{(+)} (z_1,\ldots, z_n)=z_k\overline{z_j}+ \overline{z_k}z_j\\
  	\ll k,j\gg^{(-)} (z_1,\ldots, z_n)=i(z_k\overline{z_j}- \overline{z_k}z_j)
  	\end{matrix}
  	\]
  \end{definition}

  Clearly each of this functions define a hermitian quadric form with signature $(1,n-2,1)$.  
   The following results  which  we present without proof will be
   useful. First  let us  consider the following stability  lemma. 
   
   \begin{lemma}[Weyl, see \cite{bha}] \label{weyl}
   	Let $A,B\in M(n,\C)$ be hermitian matrices with eigenvalues $\alpha_1,\ldots, \alpha_n$ and 
   	$\alpha_1,\ldots, \alpha_n$ respectively. Then 
   	\[
   	max_j \vert \alpha_j-\beta_j \vert \leq \parallel A- B\parallel. 
   	\]
   	Where $\parallel \mbox{ } \parallel $ denotes the operator bound norm.
   \end{lemma}
   
   And now let us consider the following result.

\begin{lemma} \label{l:sig}
Let $A\in M(n,\C)$ be an invertible  matrix, let us define:
\[
C
=
\left (
\begin{array}{ll}
0 & A\\
 \bar{A}^t & 0\\
\end{array}
\right ),
\]
 then $(x,y)\in \C^n\times \C^n$ is an eigenvector of $C$  with eigenvalue $\lambda $  if an only if $x$ is and  eigenvector of $A  \bar{A}^t$ with eigenvalue $\lambda^2$ and $y=\lambda A x$. 
\end{lemma}

\begin{lemma} \label{l:h1}
If   $n=2k+1$ and  $\mathcal{Q}=\{Q_r\}_{r\in \R}$ is   the family of hermitian quadratic  forms given by:
\begin{scriptsize}
\[
r\ll n,n\gg ^{(+)
}+\sum_{j=0}^{k-1} (-1)^{j+k} \ll 1+j,n-j\gg^{(+)}+\sum_{j=0}^{k-2}\sum_{m=0}^j\sum_{l=0}^m(-1)^{m-j+k}\left( \begin{array}{l} m\\l\end{array} \right ) \ll 
2+j-l,n-j+m\gg^{(+)}  
\]
\end{scriptsize}
\begin{scriptsize}
\[
+\frac{1}{2}\sum_{m=1}^{k}\sum_{j=0}^{m-1}(-1)^{2k-m}\left( \begin{array}{c} m-1\\l\end{array} \right ) \ll k+1-j,k+1+m\gg ^{(+)}  
+\frac{1}{2}\ll k+1,k+1\gg ^{(+)}. 
\]
\end{scriptsize}
Then  the following properties  are valid:
\begin{enumerate}
\item \label{i:sig12}It is verified that  $e_1\in \bigcap_{ Q\in \mathcal{Q}} Q^{-1}(0)$. 
\item \label{i:sig22}For each pair  of distinct elements $r,s\in \Bbb{R}$ it holds that  
$$Q_r^{-1}(0)\cap Q_s^{-1}(0)  \subset \mathcal{L}= \langle \langle {e_1},\ldots, e_{n-1}\rangle \rangle.$$
\item \label{i:sig32}It yields that  $\bigcup_{r\in\R} Q^{-1}_r(0)\setminus \mathcal{L}=\C^n\setminus \mathcal{L}$.
\item \label{i:sig42} Each quadric in $\mathcal{Q}$ has signature $(k,k+1)$.
\item \label{i:sig52}If $A$ is a $n\times n$-Jordan block with proper value 1, then for each $r\in  \R$ we get  $A(Q^{-1}_r(0))=Q^{-1}_r(0)$ .
\end{enumerate} 
\end{lemma}
\begin{proof}
The proofs of  (\ref{i:sig12}), (\ref{i:sig22}) and (\ref{i:sig32}) are straightforward calculations so we will omit they here.\\

Let us show (\ref{i:sig42}).  A simply inspection reveals that the Hermitian    matrix  $C_r$ associated  to  $Q_r$ has the form:
\[
\left (
\begin{array}{lll}
0         & 0         & A\\
0         & 1        & b\\
\bar{A}^t & \bar{b}^t & B_r\\
\end{array}
\right )
\]
 where $A\in SL(k,\C)$,   $b\in \C^k$ and 
\[
B_r=
\left (
\begin{array}{llll}
0        & \dots & 0      & 0\\
\vdots   &       & \vdots &\vdots\\
0        &\dots  & 0      & 0\\
0        &\dots  & 0      & r\\
\end{array}
\right ).
\]
 Now  let us consider the following hermitian matrix
\[
C_{0,0}=
\left (
\begin{array}{lll}
0         & 0         & A\\
0         & 1         & 0\\
\bar{A}^t & 0         & 0\\
\end{array}
\right ),
\]
then lemma \ref{l:sig} yields that $C_{0,0}$ signature $(k,k+1)$.  Consider the  family  of hermitian  matrices $\{C_{{r,v}}:(r,v)\in \R\times \C^k \}$ given by
\[
C_{r,v}=
\left (
\begin{array}{lll}
0         & 0          & A\\
0         & 1          &v\\
\bar{A}^t & \bar {v}^t & B_r\\
\end{array}
\right )
\]
A straightforward calculation shows that $det(C_{r,v})\neq 0$, thus  
 Lemma \ref{weyl}, yields  that the sets $$U_{l,m}=\{(r,v)\in \Bbb{R}\times \C ^k: C_{r,v} \textrm{ has signature } (l,m)\}$$  form an open cover of disjoint sets for $\R \times \C^k$. Since $\R\times \C^k$ is connected we conclude that $U_{k,k+1}=\Bbb{R}\times \C^k$, which concludes the proof.\\

 Let us show (\ref{i:sig52}). Clearly it is enough to show that $H=(-1)^{k}Q_0$ is invariant under  $A$. The proof is by induction on $k$. If $k=1$, then 
\begin{small}
$$
\begin{array}{ll}
A H &=((z_1+z_2)\bar{z}_3+z_3(\bar{z}_1+\bar{z}_2))-\vert z_2+z_3 \vert^2+\frac{1}{2}((z_2+z_3)\bar{z}_3+z_3(\bar{z}_2+\bar{z}_3))\\	
&= (z_1\bar{z}_3+z_3\bar{z}_1)+(z_2\bar{z}_3-z_3\bar{z}_2) -\vert z_2\vert^2 -\vert z_3\vert^2-(z_2\bar{z}_3+z_3\bar{z}_2)+\frac{1}{2}(z_2\bar{z}_3+z_3\bar{z}_2)+\vert z_3\vert^2 \\	
&=H
\end{array}
$$  
\end{small}    
Let us prove the case $k_0=k$. Trivially,  we can write down $H=H_1+H_2$ where:
\begin{small}
\[
H_1=\sum_{j=1}^{k-1} (-1)^{j} \ll 1+j,n-j\gg^{(+)}+\sum_{j=1}^{k-2}\sum_{m=0}^{j-1}\sum_{l=0}^m(-1)^{m-j}\left( \begin{array}{l} m\\l\end{array} \right ) \ll 
2+j-l,n-j+m\gg^{(+)} 
\]
\[
+\frac{1}{2}\sum_{m=1}^{k-1}\sum_{j=0}^{m-1}(-1)^{2k-m}\left( \begin{array}{c} m-1\\l\end{array} \right ) \ll k+1-j,k+1+m\gg ^{(+)}  
+\frac{1}{2}\ll k+1,k+1\gg ^{(+)}. 
\]
\end{small}
\begin{scriptsize}
\[
H_2= \ll 1,n \gg^{(+)}+\sum_{j=2}^k\sum_{l=0}^{j-2}\left( \begin{array}{c} j-2\\l\end{array} \right )\ll l+2,n\gg^{(+)}+\frac{1}{2}\sum_{j=0}^{k-1}\left( \begin{array}{c} k-1\\l\end{array} \right ) \ll 2+j,n\gg ^{(+)}  .\\
\]
\end{scriptsize}
By applying  the inductive to $H_1$ we get 
\begin{scriptsize}
\[
AH_1=H_1-\ll 2,n\gg^{(+)}- \sum_{j=2}^k\sum_{l=0}^{j-2}\left( \begin{array}{c} j-2\\l\end{array} \right )\ll l+3,n\gg^{(+)}-\frac{1}{2}\sum_{j=0}^{k-1}\left( \begin{array}{c} k-1\\l\end{array} \right ) \ll 3+j,n\gg ^{(+)}
\]
\end{scriptsize}
 a simple calculation shows that: 
 \begin{scriptsize}
\[
AH_2=H_2+\ll 2,n\gg^{(+)}+ \sum_{j=2}^k\sum_{l=0}^{j-2}\left( \begin{array}{c} j-2\\l\end{array} \right )\ll l+3,n\gg^{(+)}+\frac{1}{2}\sum_{j=0}^{k-1}\left( \begin{array}{c} k-1\\l\end{array} \right ) \ll 3+j,n\gg ^{(+)}
\]
\end{scriptsize}
To conclude is enough to observe that $AH=AH_1+AH_2=H_1+H_2$. 
\end{proof}

Through similar arguments one can show

\begin{lemma} \label{l:h2}
If   $n=2k$ and  $\mathcal{Q}=\{Q_r\}_{r\in \R}$ is   the family of hermitian quadratic  forms given by:

\begin{scriptsize}
\[
r\ll n,n\gg^{(+)}  +\sum_{j=0}^{k-1} (-1)^{j} \ll 1+j,n-j\gg^{(-)}+\sum_{j=0}^{k-2}\sum_{m=0}^j\sum_{l=0}^m(-1)^{m-j}\left( \begin{array}{l} m\\l\end{array} \right ) \ll 
2+j-l,n-j+m\gg^{(-)}.  
\]
\end{scriptsize}
Then the following properties hold:
\begin{enumerate} 
\item \label{i:sig1} It is verified that  $e_1\in \bigcap_{ Q\in \mathcal{Q}} Q^{-1}(0)$. 
\item \label{i:sig2} For each pair  of distinct elements $r,s\in \Bbb{R}$ it holds that  
$$Q_r^{-1}(0)\cap Q_s^{-1}(0)  \subset  \mathcal{L}=\langle \langle {e_1},\ldots, e_{n-1}\rangle \rangle.$$
\item \label{i:sig3} It yields that  $\bigcup_{r\in\R}Q^{-1}_r(0)\setminus \mathcal{L}=\C^n\setminus \mathcal{L}$.
\item \label{i:sig4} Each quadric in $\mathcal{Q}$ has signature $(k,k)$.
\item  \label{i:sig5} If $A$ is the $n\times n$-Jordan block being  $1$ its proper value, then  for each $r\in \Bbb{R}$ we have $A(Q^{-1}_r(0))=Q^{-1}_r(0)$.
\end{enumerate} 
\end{lemma}

Clearly, we get similar 
  results if we replace the matrix $A$ by any  Jordan-block with an  unitary eigenvalue. Now, let us provide a characterization of parabolic elements in terms of foliations by $(k,l)$-spheres, the following  will be helpful.

  \begin{lemma} \label{l:parwed}
  	Let $\gamma\in PSL(n+1,\C)$ be an element such that $\gamma$ has a lift $\widetilde \gamma$ such
  	that $\widetilde\gamma$ is a $(n+1)\times( n+1)$-Jordan block being 1 its unique eigenvalue, then
  	\begin{enumerate}
  		\item \label{i:parwed1} For each $n> k\geq  0$ the action of $\left [\bigwedge_k\widetilde\gamma\right]$ on $\Bbb {P}(\bigwedge^{k}_{j=1} \C^{n})$ has a unique fixed point in $\iota(Gr(k-1,n))$ namely
  		$$[e_1\wedge e_2\wedge \cdots \wedge e_k].$$
  		\item \label{i:parwed2} For each
  		$\ell\in Gr(k,n)$  we know  $\gamma^m(\ell)
  		\xymatrix{ \ar[r]_{m \rightarrow \infty}&} \langle\langle [e_1],\ldots,[e_{k+1}]\rangle\rangle$.
  		
  		\item \label{i:parwed2.5} It is verified that $\P^{n-1}\setminus Eq(\langle \gamma\rangle )=\langle\langle [e_1],\ldots,[e_{n}]\rangle\rangle$.
  		
  		\item \label{i:parwed3} It is verified that $\Lambda_{Kul}(\langle \gamma\rangle )=\langle\langle [e_1],\ldots,[e_n]\rangle\rangle$.
  		\item \label{i:parwed4} If $\gamma$ is conjugated to an element in $PU(k,l)$, then $k=floor((n+1)2^{-1})$ and $l=ceiling((n+1)2^{-1})$, here $floor(x)$ denote the largest integer less than or equal to $x$  and  $ceiling(\cdot)$ is the smallest integer greater than or equal to $x$.
  		  	\end{enumerate}
  \end{lemma}
  \begin{proof}
  	 Let us show (\ref{i:parwed1}). Let us make the	 proof by induction on $n$. For $n=1$, we are simply considering the action of
  	$
  	\left(
  	\begin{array}{ll}
  	1 &1\\
  	0 & 1\\
  	\end{array}\right)$
  	on $\bigwedge^{1} \C^{1+1}=\C^2$, which trivially has a unique fixed point, thus in this case the proof is trivial. Now let us proceed to check the case $n+1$. Let $\mathcal{L}$ be a $k$-dimensional projective invariant under the action of $\widetilde \gamma$, from standard linear algebra we may can assume that $k>0$ . Then $\mathcal{W}=\mathcal{L}\cap \langle\langle [e_1],\ldots, [e_{n_0}]\rangle\rangle$ is a projective subspace, whose dimension is either $k-1$ or $k-2$, invariant under $\widetilde \gamma$. By the inductive hypothesis we conclude that $\langle\langle [e_1],\ldots,[e_{k-1}]\rangle\rangle\subset \mathcal{L} $. Therefore, given $p\in \mathcal{L}\setminus \langle\langle [e_1],\ldots,[e_{k-1}]\rangle\rangle$ we have:
  	\[
  	p=\alpha_1 e_1+\ldots \alpha_{n+1} e_{n+1}
  	\]
  	Let $m=max\{j=1,\ldots n+1:\alpha_j\neq 0  \}$, then $m>k-1$. The set  $\{p,\widetilde{ \gamma}(p),\ldots, \widetilde{ \gamma}^{m-1}(P)\}$ is linearly independent, this can be viewed from:
  
  	\[
  	\begin{vmatrix}
  	\alpha_1      & \alpha_1 +\alpha_2& \ldots & \alpha_{1}+\alpha_m  \\
  	\alpha_2     & \alpha_2 +\alpha_3 & \ldots  & \alpha_2+0\\
  	\vdots        & \vdots                       &  \ddots & \vdots\\
  	\alpha_{m} & \alpha_m +0                 &                & \alpha_m +0\\
  	\end{vmatrix}
  	=
  	\begin{vmatrix}
  	\alpha_1      & \alpha_2& \ldots & \alpha_m  \\
  	\alpha_2     & \alpha_3 & \ldots  & 0\\
  	\vdots        & \vdots                       &  \ddots & \vdots\\
  	\alpha_{m} &  0                 &                & 0\\
  	\end{vmatrix}
  	\neq 0
  	\]
  That is $m\leq k$, which  concludes this part of the proof.

  	The proof of (\ref{i:parwed2}) follows from part (\ref{i:parwed1}) of the present Lemma and Theorem \ref{l:lila}.
  	The proof of (\ref{i:parwed2.5}) is trivial from Section \ref{s:equi}.
  	The proof of (\ref{i:parwed3}) follows from parts (\ref{i:parwed1}), (\ref{i:parwed2.5})  of the present Lemma and Theorem  \ref{p:pkg}.\\
  	
  Finally let us shows (\ref{i:parwed4}).   	Let us suppose that $\gamma$ is conjugate to an element in $PU(k,l)$ with $l-k\geq 2.$
	Let $(\gamma_{m})_{m\in\mathbb{N}}$ a sequence of different elements of $\Gamma=\langle \gamma\rangle.$ Taking  the polar unitary decomposition of  $\gamma_m$ we get 
	\begin{equation}
	\gamma_m = u_m\begin{bmatrix}
	D_1^{(m)} & 0 & 0 \\ 
	0 & I_{l-k} & 0 \\ 
	0 & 0 & D_2^{(m)}
	\end{bmatrix} v_m,
	\end{equation}
	with $D_1=Diag(a_{1,m}, a_{2,m}, \cdots, a_{k,m}),$ $D_2=Diag(a_{k,m}^{-1}, a_{k-1,m}^{-1}, \cdots, a_{1,m}^{-1})$,  $a_{i,m}\geq a_{j,m}\geq 1$ if $i>j$, $I_{l-k}$ is the identity in $GL(l-k,\mathbb{C})$ and  $u_m, v_m$ are unitary matrices.  Without lost of generality, let us assume that $a_{k,m} \xymatrix{ \ar[r]_{m \rightarrow  \infty}&}\infty$.	
	Taking a subsequence $(\gamma_{n_m})$ of $(\gamma_m)$,  if necessary, we may assume that there are  $u,v$ unitary matrices with $u_{n_m} \xymatrix{ \ar[r]_{m \rightarrow  \infty}&} u$ and $v_{n_m} \xymatrix{ \ar[r]_{m \rightarrow  \infty}&}v.$
	Let us define the following subspaces
	 \[
	 \begin{matrix}
	 J_1:=u(\langle\langle e_{k+1},e_{k+2},\cdots,e_l\rangle\rangle),\\
	 J_2:=v^{-1}(\langle\langle e_{k+1},e_{k+2},\cdots,e_l\rangle\rangle),\\
	 J_3:=u(\langle\langle e_1 ,e_2 ,\cdots,e_k \rangle\rangle),\\
	 J_4:=v^{-1}(\langle\langle e_1 ,e_2 ,\cdots,e_k \rangle\rangle)
	 \end{matrix}
	 \]
	  and consider  the function 
	\[\begin{array}{cccc}
	\phi: & J_1 & \rightarrow & J_2 \\
	& p & \mapsto & uv(p).
	\end{array}\]
	Let   $\mathcal{L}\subset \mathbb{P}^n_{\mathbb{C}}$ be any $k$ dimensional projective subspace satisfying $J_4\subset \mathcal{L}\subset \langle\langle J_4, J_2\rangle\rangle,$ then there is a element $p\in J_2$ such that $\mathcal{L}=\langle\langle J_4,p\rangle\rangle.$
If  $p=v^{-1}\left [\sum_{j=k+1}^l x_{j}e_j\right ]$, define 
	$$x_m=v^{-1}([a_{1,n_m}^{-2}x_1:a_{2,n_m}^{-2}x_2:\cdots:a_{k,n_m}^{-2}x_k:x_{k+1}:\cdots:x_l:\cdots:0]),$$ clearly $(x_m)_{m\in\mathbb{N}}\subset \mathcal{L}$  and  
	 \begin{equation}\label{e:alex}
	\gamma_{n_m}(x_{n_m}) \xymatrix{ \ar[r]_{m \rightarrow  \infty}&} \phi(p).
	\end{equation}
Since  $(\gamma_{n_m}(J_4))$ converge to $J_3$, the      Grassmannian is compact  and Equation \ref{e:alex}, we deduce   that $\gamma_{n_m}(\mathcal{L})  \xymatrix{ \ar[r]_{m \rightarrow  \infty}&} \langle\langle J_3,\phi(q)\rangle\rangle,$ however this is a contradicts part (\ref{i:parwed2}) of the present lemma,  so $l-k<2.$
	\end{proof}
	The last part of the previous lemma is telling us that the Hermitian metric constructed in  Lemmas  \ref{l:h1} and \ref{l:h2} , is ``essentially" unique. Now we can show: 
\subsection*{Proof of Theorem \ref{t:p1}}
\begin{proof}
	Let $(k, \{V_j\}_{j=1}^k, \{\gamma_j\}_{j=1}^k)$ be  a Jordan decomposition for $\widetilde \gamma$.  For each $j:1,\ldots k$ and $r\in \Bbb{R}$,  let $Q_{j,r}$ be   the hermitian quadratic form in $V_j$ given  by: 
	\begin{scriptsize}	
	\[
	Q_{j,r}=
	\left  \{
	\begin{array}{ll}
	\textrm{The   corresponding quadratic  given   by Lemmas  \ref{l:h1} and   \ref{l:h2} } & \textrm{ if } \gamma_j \textrm{ is a Jordan block}\\
	\textrm{The  quadric induced by  } e^r Id_{V_j}  &\textrm{ if } \gamma_j  \textrm{ is a diagonal matrix}
	\end{array}
	\right.
	\]
	\end{scriptsize}
	Let $(k_i,l_i)$ be the signature of the quadric $Q_{j,r}$.
	Then  $T_r=\bigoplus_{j=1}^k Q_{j,r}$ is a family of $\gamma$-invariant  hermitian quadratics each of one has signature 
	$\left (\sum_{j=1}^k k_j,\sum_{j=1}^k l_j \right )$. For each  $j$ define 
	\[
	Z_j=\left 
	\{
	\begin{array}{ll}
	\textrm{The unique fixed point of } \gamma_j,  \textrm{ if } \gamma_j  \textrm{ is  non-diagonalizable}\\
	\emptyset  \textrm{ in other case}\\   
	\end{array}
	\right.
	\]
	\[
	K_j=\left 
	\{
	\begin{array}{ll}
	\textrm{The unique  fixed hyperplane  of } \gamma_j,  \textrm{ if } \gamma_j  \textrm{ is  non-diagonalizable}\\
	V_j\textrm{ in other case}\\   
	\end{array}
	\right.
	\]
	Define $\mathcal{Z}=\left \langle \left \langle \bigcup_{j=i}^k Z_j\right \rangle \right  \rangle $ and $\mathcal{K}=\left \langle\left  \langle \bigcup_{j=i}^k H_j\right \rangle\right  \rangle $, then $\mathcal{Z}^\bot=\mathcal{K}$, $\gamma\vert_ \mathcal{Z}^\bot$ is non-loxodromic, 	$Fix(\gamma)\subset \mathcal{Z}^\bot$, $ \bigcup_{T\in \mathcal{F}} T\setminus \mathcal{Z}^\bot=\P^n\setminus \mathcal{Z}^\bot$ and whenever $T_{1},  T_{2}$ are different quadratics in $T_r$ we have   $\emptyset \neq \mathcal{ Z}\subset T_1\cap T_2\subset \mathcal{Z}^\bot $.	Which concludes this part of the  proof.\\
	
	Let us show the reciprocal proposition. Let $\gamma\in PSL(n+1,\C)$ and   $k,l\in \Bbb{N}$,       $\mathcal{F}$ a family of $\gamma$-invariant $(k,l)$-spheres and $\mathcal{Z}$ a projective  subspace of $ \Bbb{P}^n_{\Bbb{C}}$ satisfying   the items of this Theorem.\\

	Claim  1.- The element $\gamma$ can not be elliptic. On the contrary, let us assume that $\gamma$ is elliptic, by Proposition  \ref{p:eli} there is a set $V\subset \H^{k,l}_\C$ with $k$ elements  and a set $W\subset \P^n\setminus  \H^{k,l}_\C$ with $l$ elements such that  each point in $V \cup W$  is fixed by  $\gamma$ and $\langle \langle V\cup W \rangle \rangle=\P^n $, which is a contradiction since $Fix(\gamma)\subset \mathcal{Z}^{\bot}$.\\
	
	Let us assume that $\gamma$ is loxodromic,  by results in section \ref{s:equi}   there are  pseudo-projective transformations $\tau,\rho$   and a sequence $(n_m)_{m\in \Bbb{N}}\subset \Bbb{N}$ such that $\gamma^{n_m} \xymatrix{ \ar[r]_{m \rightarrow  \infty}&} \tau$, $\gamma^{-n_m} \xymatrix{ \ar[r]_{m \rightarrow  \infty}&} \rho$,  $Im(\rho)\cap Im(\tau)=\emptyset$, $Im(\rho),Im(\tau)$ are $\gamma$-invariant projective sub spaces and $\gamma\mid_{Im(\rho)},\gamma\mid_{Im(\tau)}$ are elliptic, thus $Im(\rho)\cup Im(\tau)\subset\mathcal{Z}^\bot$.  On the other hand, given  $p\in \Bbb{P}^{k+l-1}_{\Bbb{C}}\setminus (Ker(\rho)\cup Ker(\tau)\cup\mathcal{Z}^\bot)$, we get  
	$\gamma^{n_m}(p) \xymatrix{ \ar[r]_{m \rightarrow  \infty}&} \tau(p)\in \mathcal{Z}^\bot$ and $\gamma^{-n_m}(p) \xymatrix{ \ar[r]_{m \rightarrow  \infty}&} \rho (p)\in \mathcal{Z}^\bot$, therefore $\gamma$ restricted to $\mathcal{Z}^\bot$ is  loxodromic element, which is a contradiction.
\end{proof}

The following example shows a  non parabolic transformation preserving a family of invariant horospheres sharing a tangent hyperplane.

\begin{example}\label{e:parfix}
	Consider the transformation given by:
	\[
	\tau=
	\begin{pmatrix}
	i \\
	&i\\
	&&-1
	\end{pmatrix},
	\]
	for each $r\in \Bbb{R}$  define the sphere:
	$$
	C_r
	=
	\{
	[z_1, z_2, z_3]: 
	r\vert z_2\vert ^2 + \vert z_3\vert^2 + 
	\overline{z_1}z_2 +z_1\overline{z_2} = 0, r \in \Bbb{R}
	\}
	$$  and let $\mathcal{F}=\{C_r:r\in \Bbb{R}\}$. An easy computation prove that $\mathcal{F}$ is a family satisfying:
	\begin{enumerate} 
		\item Each element of $\mathcal{F}$ is a $(1,2)$-sphere which is $\tau$-invariant.
		\item We have  $C_r\cap C_s=[e_1]$ whenever $s\neq r$.
		\item Also $\bigcup\mathcal{F}\setminus {[e_1]}=\Bbb{P}^2_{\Bbb{C}}\setminus\overleftrightarrow{e_1,e_3}$
	\end{enumerate}
\end{example}
\subsection{The  limit set for parabolic elements}
In this part of the article we will describe the Kulkarni's limit set for cyclic groups generated by a parabolic element. Let us start with a couple of  examples which shows us that the Kulkarni's discontinuity set can  or  cannot agree  with the equicontinuity  set.
  
\begin{example}\label{e:ekdp1}
	Consider first  the matrix $S$:
	\begin{displaymath}
	S=
	\left (
	\begin{array}{ll}
	B& 0\\
	0 & B \\
	\end{array}
	\right )
	\end{displaymath}
	where $B$ is $k\times k$-Jordan block:
	\begin{displaymath}
	B=
	\left (
	\begin{array}{cccccc}
	1 & 1 & 0 \\
	0 & 1 & 1 \\
	0 & 0 & 1\\
	\end{array}
	\right )  \;.
	\end{displaymath}
	Then one can check
	$Eq(\langle S \rangle)=\Omega_{\rm Kul} (\langle S \rangle)$ and their complement consists of the projective space generated by the set  $ \{e_1,e_2, e_4,e_5\}$. 
\end{example}
\begin{example} \label{e:ekdp2}
	Now let 
	\begin{displaymath}
	C=
	\left (
	\begin{array}{ccc}
	B & 0\\
	0 & D \\
	\end{array}
	\right )\;,
	\end{displaymath}
	where  $B$ is as above and $D$ is:
	
	\begin{displaymath}
	D= \left (\begin{array}{cccc}
	1 & 1 & 0 & 0\\
	0 & 1 & 1 & 0\\
	0 & 0 & 1 & 1\\
	0 & 0 & 0 & 1\\
	\end{array}
	\right ) \;.
	\end{displaymath}
	In this case one has  that $\Omega_{\rm Kul} (\langle C \rangle)$ is the complement of the projective space generated by the set $ \{e_1,e_2,e_4, e_5,e_6\}$, while the  complement of the equicontinuity region is the projective space generated by the set   $\{e_1,e_2,e_3,e_4, e_5,e_6\}$. 
\end{example}

The previous examples are telling  us that in order to understand the Kulkarni's limit set we should be interested  when the transformation is made of two blocks, {\it  i. e. } we should consider the following results:

\begin{lemma} \label{l:par11}
Let $A\in GL(k,\C)$ be a diagonal matrix such that each of its proper values  is a unitary complex number,  $B\in GL(l,\Bbb{C})$ be a $l\times l$-Jordan block such that $1$ is its proper value  and    $\widetilde \gamma\in GL(k+l,\C)$ be given by: 
\[
\widetilde \gamma=
\left ( 
\begin{array}{ll}
A & 0\\
0 &  B\\
\end{array}
\right ).
\]
 If $\gamma=[[\widetilde\gamma]]$, then:
  \[
\Omega_{Kul}(\langle  \gamma \rangle)=Eq(\langle  \gamma \rangle)=
\Bbb{P}^{k+l-1}_\Bbb{C}\setminus\langle \langle[e_1],\ldots,[e_{k+l-1}]\rangle \rangle
\]
\end{lemma}
\begin{proof}
The proof is by  induction on $l$. For $l=2$ we get that 
\[
\widetilde {\gamma}=
\left (
\begin{array}{lll}
B &   &\\
  & 1 & 1 \\
  & 0 & 1\\
 \end{array}
\right ),
\]
for which the   claim is trivial. Let us show the claim in the case $l+1$. Set  $\langle \langle [e_1], \ldots, [e_{k+l-1}] \rangle\rangle=\mathcal{L}_1$,  it will be enough to show that for $z\in\mathcal{L}_1\setminus  \langle \langle [e_1], \ldots, [e_{k}] \rangle\rangle$,  it holds that $\langle \langle z, e_{k+l}\rangle \rangle \subset \Lambda_{Kul}(\langle \gamma\rangle )$. Applying  the inductive hypothesis  to $ \gamma$ restricted to 
$\mathcal{L}_2=\langle \langle [e_1], \ldots, [e_{k+l}] \rangle\rangle$,  we conclude that $\Lambda_{Kul}(\langle \gamma\vert_{\mathcal{L}_2}\rangle)=\mathcal{L}_1$. Thus there is a sequence $(z_{1m})\subset \mathcal{L}_2$ such that the cluster points of $(z_{1m})$ lies on $\P^{k+l}\setminus \Lambda( \langle \gamma   \rangle) $, and  $$\gamma^m (z_{1m})  \xymatrix{ \ar[r]_{m \rightarrow  \infty}&} z. $$
Since $\mathcal{L}_2$ is compact we can assume $z_{1m}  \xymatrix{ \ar[r]_{m \rightarrow  \infty}&} z_1$, where $z_1=[v_1,\ldots,v_{k+l},0]$
 and $\sum_{j=1}^l \vert v_{k+j} \vert^2\neq 0 $. On the other hand, let $\langle\langle [e_{k+1}], \ldots, [e_{k+l+1}]\rangle\rangle=\mathcal{L}_3$ and $M\subset \mathcal{L}_3$ a projective subspace of dimension $l-1$  such that $[e_{k+1}],\left [\sum_{j=1}^lv_{k+j} e_{k+j} \right ]\notin \mathcal{M} $,  by applying  Lemma 
\ref{l:parwed},  to $\gamma$ restricted to $\mathcal{L}_3$ and $\mathcal{M}$ we conclude that there is a sequence $(z_{2m})\subset  \mathcal{M}$ such that: 
$$\gamma^m (z_{2m})  \xymatrix{ \ar[r]_{m \rightarrow  \infty}&} [e_{k+l}]. $$ 
Clearly we can assume that $z_{1m}\neq z_{2m}$ for each $m$. Define $\ell_m=\langle \langle k_{1m},k_{2m}\rangle \rangle $, thus 
 $$
 \gamma^m(\ell_m)  \xymatrix{ \ar[r]_{m \rightarrow  \infty}&} \langle \langle z,e_{k+l} \rangle \rangle,
 $$
in order to conclude the proof we must observe  that the cluster points of $\bigcup_{m\in \Bbb{N}} \ell_m$ do not intersect $ \Lambda(\langle \gamma\rangle)$, which concludes the proof.
\end{proof}

\begin{lemma} \label{l:par22}
Let   $B\in GL(l,\Bbb{C})$ be a $l\times l$-Jordan block such that $1$ is its proper value  and    $\widetilde \gamma\in GL(2l,\C)$ be given by: 
\[
\widetilde \gamma=
\left ( 
\begin{array}{ll}
B & 0\\
0 &  B\\
\end{array}
\right ).
\]
If $\gamma=[[\widetilde\gamma]]$, then:
\[
\Omega_{Kul}(\langle  \gamma \rangle)=Eq(\langle  \gamma \rangle)=
\Bbb{P}^{k+l-1}_\Bbb{C}\setminus\langle \langle[e_1],\ldots,[e_{l-1}], [e_{l+1}],\ldots,[e_{2l-1}]\rangle \rangle
\] 
\end{lemma}

\begin{proof} Clearly 
$$
Eq(\langle \gamma\rangle)=\Bbb{P}^n_\Bbb{C}\setminus
\langle \langle[e_1],\ldots,[e_{l-1}], [e_{l+1}],\ldots,[e_{2l-1}]\rangle \rangle
$$
 By  Theorems \ref{t:equi} and \ref{p:pkg} to conclude the result will be enough to show that 
\[
W=
\langle \langle[e_1],\ldots,[e_{l-1}], [e_{l+1}],\ldots,[e_{2l-1}]\rangle \rangle
\subset L_2(\langle \gamma\rangle).
\]
  Let 
 $\ell_1\subset\langle  \langle[e_1],\ldots,[e_{l}]\rangle\rangle \setminus\{[e_1]\}$  and $\ell_2\subset\langle  \langle[e_{l+1}],\ldots,[e_{2l}]\rangle\rangle \setminus\{[e_{l+1}]\}$ be projective subspaces of dimension $l-2$. Define 
$
\mathcal{L}=\left \langle \left \langle \ell_1\cup  \ell_2 \right \rangle \right \rangle,
$  
clearly  $\mathcal{L}\cap \Lambda(\langle \gamma\rangle )=\emptyset$. By Lemma \ref{l:parwed} we get 
\[
\gamma^m(\mathcal{L}) 
\xymatrix{ \ar[r]_{m \rightarrow  \infty}&} 
W.
\]
Which concludes the proof.
\end{proof}

In order  to determine the limit set for parabolic transformation the previous analysis is not  enough, we need to introduce the following notions.
\begin{definition}
Let  $M\in SL(n+1,\Bbb{C})$ and  $\lambda_1\geq \lambda_2 \geq\ldots \geq \lambda_n >0 $ be the proper values of  $M\bar{M}^t$, we define the    $(k,p)$-Ky Fan norm, see \cite{kyfan}, by:

\[
\vert M\vert_{k,p}=\left (\sum_{j=1}^k  \lambda_j^{\frac{p}{2}} \right )^{\frac{1}{p}}  
\]

\end{definition}

 \begin{lemma}\label{l:dym} [See Lemma 10.9 in \cite{dym}]
 Let   $A\in SL(n+1,\Bbb{C})$, then:
 \[
 \vert A \vert_{1,1}^2 (\vert A \vert_{2,1}-\vert A \vert_{1,1})^2 
 =
 max
 \{
 det(\bar{V}^{t} A \bar{A}^t V): V\in M(n\times 2,\Bbb{C}) \textrm{ and } \bar{V}^{t}V=Id_{2}
 \}.
 \]
 \end{lemma}

\begin{lemma}\label{l:kyfan} Let  $M\in SL(n+1,\Bbb{C})$ be a Jordan block of size  $n+1\times n+1$. If 1 is the unique proper value of  $M$, then:

 \begin{enumerate}
 \item \label{i:1ky} There are  $0<s<r$ such that for every $m\in \Bbb{N}$ we get:
 \[
s 
\left (
 \begin{matrix}
 m\\
 n
 \end{matrix}
\right )
<
 \vert M^m\vert_{1,1}
 <
 r
  \left (
 \begin{matrix}
 m\\
 n
 \end{matrix}
\right ).
 \]
 
 \item \label{i:2ky}  Also we have:
 \[
  \frac
 {
  \vert M^m\vert_{2,1}-  \vert M^m\vert_{1,1}
 }
 {
 m
  }
 \xymatrix{ \ar[r]_{m \rightarrow  \infty}&} 
0.
 \]
 
 \end{enumerate} 
\end{lemma}
\begin{proof}
  An inductive argument shows:
\[
M^m=
\begin{pmatrix}
1 & \begin{pmatrix}m\\ 1\end{pmatrix} & \begin{pmatrix}m\\ 2\end{pmatrix}&
\ldots & \begin{pmatrix}m\\ n\end{pmatrix}\\
0 & 1 &  \begin{pmatrix}m\\ 1\end{pmatrix} & 
\ldots & \begin{pmatrix}m\\ n-1\end{pmatrix}\\
0 & 0 & 1 &   
\ldots & \begin{pmatrix}m\\ n-2\end{pmatrix}\\
 &  &  &   
 \ddots& \vdots\\
 &  &  &   
 & 1
\end{pmatrix}.
\] 
Define $B_m=M^m(\bar{M}^m)^t$,  a simple  computation shows $B_m$ is given by:
\begin{tiny}
\[
\begin{pmatrix}
\sum_{j=0}^n \begin{pmatrix}m\\ j\end{pmatrix}^2 &
\sum_{j=1}^n \begin{pmatrix}m\\ j\end{pmatrix} \begin{pmatrix}m\\ j-1\end{pmatrix} 
& 
\sum_{j=2}^n \begin{pmatrix}m\\ j\end{pmatrix} \begin{pmatrix}m\\ j-2\end{pmatrix} & \ldots
&
\sum_{j=n}^n \begin{pmatrix}m\\ j\end{pmatrix} \begin{pmatrix}m\\ j-n\end{pmatrix}
\\
\sum_{j=1}^n \begin{pmatrix}m\\ j\end{pmatrix} \begin{pmatrix}m\\ j-1\end{pmatrix}  & \sum_{j=0}^{n-1} \begin{pmatrix}m\\ j\end{pmatrix}^2 &  \sum_{j=1}^{n-1} \begin{pmatrix}m\\ j\end{pmatrix}  \begin{pmatrix}m\\ j-1\end{pmatrix}& 
\ldots &\sum_{j=n-1}^{n-1} \begin{pmatrix}m\\ j\end{pmatrix}  \begin{pmatrix}m\\ j-n+1\end{pmatrix}\\
\sum_{j=2}^{n} \begin{pmatrix}m\\ j\end{pmatrix}  \begin{pmatrix}m\\ j-2\end{pmatrix} & \sum_{j=1}^{n-1} \begin{pmatrix}m\\ j\end{pmatrix}  \begin{pmatrix}m\\ j-1\end{pmatrix} & \sum_{j=0}^{n-2} \begin{pmatrix}m\\ j\end{pmatrix}^2 &   
\ldots &\sum_{j=n-2}^{n-2} \begin{pmatrix}m\\ j\end{pmatrix}  \begin{pmatrix}m\\ j-n+2\end{pmatrix}\\
 \vdots&\vdots  &\vdots  &   
 \ddots& \vdots\\
\sum_{j=n}^{n} \begin{pmatrix}m\\ j\end{pmatrix}  \begin{pmatrix}m\\ j-n\end{pmatrix}&  \sum_{j=n-1}^{n-1} \begin{pmatrix}m\\ j\end{pmatrix}  \begin{pmatrix}m\\ j-n+1\end{pmatrix} & \sum_{j=n-2}^{n-2} \begin{pmatrix}m\\ j\end{pmatrix}  \begin{pmatrix}m\\ j-n+2\end{pmatrix} &   
 & \sum_{j=0}^0 \begin{pmatrix}m\\ j\end{pmatrix}^2 
\end{pmatrix}
\] 
\end{tiny}
 Let us show  (\ref{i:1ky}).  From the previous equation we deduce:
 
 \[
\left [\left [
\begin{pmatrix}m\\ n\end{pmatrix}^{-2}B_m 
\right ]\right ]
\xymatrix{ \ar[r]_{m \rightarrow  \infty}&} 
\left [ \left [
\begin{matrix}
1 & 0& 0 &
\ldots & 0\\
0 & 0 & 0& 
\ldots & 0\\
0 & 0 & 0 &   
\ldots &0\\
 \vdots &\vdots  &\vdots  &   
 \ddots& \vdots\\
 0 & 0 &0  &   
 \ldots& 0
\end{matrix}
\right ] \right ]
\]
Thus  there are $0<s<r$ satisfying:
 \[
s 
 \begin{pmatrix}
 m\\
 n
 \end{pmatrix}^2
<
\vert B_m\vert_{1,1} 
=
 \vert M^m\vert_{1,1}^2
 <
 r
 \begin{pmatrix}
 m\\
 n
 \end{pmatrix}^2, 
 \]
which concludes this part of the proof.\\ 
 
In order to conclude let us show  (\ref{i:2ky}). By Lemma \ref{l:dym}, there are   $V_m\in SL(n\times 2,\Bbb{C})$ such that:
 
  \[
  k_m=\vert M^m\vert_{1,1}^2(\vert M^m\vert_{2,1}^2-\vert M^m\vert_{1,1})^2=
  det (\bar{V}^t A^m\bar{A^m}^tV_m)
  \] 
  A straightforward calculation shows:
 \[
 k_m
 \left (
 m
 \begin{pmatrix}
 m\\
 n
 \end{pmatrix}
\right  )^{-2}=
det 
\left (\bar{V}^t 
\left (
 m
 \begin{pmatrix}
 m\\
 n
 \end{pmatrix}
\right  )^{-2}
 B_mV_m
 \right ) \xymatrix{ \ar[r]_{m \rightarrow  \infty}&} 
0, 
 \]
which concludes the proof.
\end{proof}

\begin{lemma} \label{l:par22a} 
	Let     $A\in GL(k,\Bbb{C})$ and $B\in GL(l,\Bbb{C})$ be Jordan blocks such that $1$ is its proper value  and $k>l$. If     $\widetilde \gamma\in GL(k+l,\C)$ is given by: 
	\[
	\widetilde \gamma=
	\left ( 
	\begin{array}{ll}
	A & 0\\
	0 &  B\\
	\end{array}
	\right ), 
	\]
	and $\gamma=[[\widetilde\gamma]]$, then:
	\[
	\begin{array}{l}
	\Lambda_{Kul}(\langle  \gamma \rangle)=
	\langle \langle[e_1],\ldots,[e_{l_k-1}], [e_{k+1}],\ldots,[e_{k+l-1}]\rangle \rangle\\
	\end{array}
		\] 
	
\end{lemma}
\begin{proof}
By Lemma \ref{l:parwed} and results in  \cite{neretin} there are  $\nu_1\in SL(k,\mathbb{C}), \nu_2\in SL(l,\mathbb{C})$,   Hermitian  matrices $H_1\in GL(k,\C), H_2\in GL(l,\C)$ and  $(u_{1m}),(v_{1m})\in U(k)$, $(u_{2m}),(v_{2m})\in U(l)$   such that: The signatures  of $H_1$ and $H_2$  are respectively  $(floor(\frac{k}{2}),ceiling(\frac{k}{2})), (floor(\frac{l}{2}),ceiling(\frac{l}{2}))$, we know  
$\widehat{A}= \nu_1 A \nu_1^{-1}\in U(H_1)\, \widehat{B}= \nu_1 B\nu_1^{-1}\in U(H_2)$ and 

%
%

\begin{equation} \label{Cartan}
\begin{array}{l}
\widehat{A}^{m}=  \hat{u}^{1m} Diag(\alpha_{1m},\cdots,\alpha_{km})\hat{v}_{1m}\\                                                           
\widehat{B}^{m}=\hat{u}_{2m }Diag(\beta_{1m},\cdots,\cdots,\beta_{lm})\hat{v}_{2m}
\end{array}
\end{equation}

where  $\alpha_i>\alpha_{i+1}$, $\beta_i>\beta_{i+1}$ and 
\[
 \begin{array}{l}
 \alpha_{floor(k/2)-i}\alpha_{ceiling(k/2)+1+i}=1 \textrm{ for } 0\leq i\leq floor(k/2)-1\\
 \beta_{floor(l/2)-i}\beta_{ceiling(l/2)+1+i}=1  \textrm{ for } 0\leq i\leq floor(k/2)-1\\
\end{array}
\]

Since the dynamics of $A,B$ and $\hat{A},\hat{B}$ are conjugate, from Lemma \ref{l:kyfan} we  get
\begin{equation}
\begin{array}{l}
\frac{\alpha_{1m}}
{\beta_{1m}}  \xymatrix{ \ar[r]_{m \rightarrow  \infty}&}\infty\\
\frac{\beta_{1m}}
	{max\{\alpha_{im}, \beta_{im}:i>2\}} \xymatrix{ \ar[r]_{m \rightarrow  \infty}&}  \infty\\
\end{array}
\end{equation}
 
Up to  conjugation,  we can assume that: 
\begin{equation}
\label{reorderedPolar}
{\gamma}^m=\hat{u}_m\begin{pmatrix}
D_m & & \\
& \beta_{1m}^{-1} & \\
& & \alpha_{1m}^{-1}
\end{pmatrix} \hat{v}_m
\end{equation}
where $\hat{u}_m$ and $\hat{v}_m$ are unitary matrices and  $$D_m=Diag(\alpha_{1m},\cdots,\alpha_{k-1,m},\beta_{1m},\cdots,\beta_{l-1,m}),$$. 
Let's denoted by $\rho_{im}$ the elements of $D_m$,  define
\[
\begin{array}{l}
\kappa_{im}=min\{\rho^{-1}_{im}\rho_{jm}:   j=1,\ldots, k+l-2 \}  \\
\end{array}
\] 
Let $(n_m)\subset \Bbb{N}$ be any  strictly  increasing sequence such that 
 $\hat{v}_m \xymatrix{ \ar[r]_{m \rightarrow  \infty}&} \hat{v}\in U(k+l)$,  $\hat{u}_m \xymatrix{ \ar[r]_{m \rightarrow  \infty}&}\hat{u}\in U(k+l)$ and $\kappa_{im}  \xymatrix{ \ar[r]_{m \rightarrow  \infty}&}\kappa_i \in \R$
  then
   \[
   \begin{array}{l}
   \hat{u}(\langle\langle{e}_1,\cdots,{e}_s\rangle\rangle)=\langle\langle {e}_1,\cdots,{e}_s\rangle\rangle,\\
    \hat{v}^{-1}(\langle\langle{e}_{s+1},{e}_{s+2}\rangle\rangle)=\langle\langle {e}_{s+1},{e}_{s+2}\rangle\rangle,\\
  \end{array}
    \]
  Let $x=[x_1,\ldots, x_{k+l-2},0,0]\in \langle\langle e_{1},\ldots, , e_{k+l-2}\rangle \rangle$
  and define
\[
\zeta_m=\hat{v}^{-1}_m([\kappa_{1m} x_1,\cdots,\kappa_{k+l-2,m} x_{k+l-2},1,1])
\]
thus 
\[  
\begin{array}{l}
 \zeta_m \xymatrix{ \ar[r]_{m \rightarrow  \infty}&} \zeta=\hat{v}^{-1}_m([\kappa_{1} x_1,\cdots,\kappa_{k+l-2} x_{k+l-2},1,1])\in \P^{n}\setminus \langle\langle {e}_{s+1},{e}_{s+2}\rangle\rangle,\\
 \gamma^{n_m}\zeta_m\rightarrow \hat {u} x.
 \end{array}
 \] 
    This shows  the assertion.
\end{proof}

The following is an easy consequence of Lemmas \ref{l:par11},  \ref{l:par22} and  \ref{l:par22a}. 

\begin{lemma} \label{l:par22}
	Let $\gamma\in PSL(n+1,\C)$ be a parabolic  element, 
	$\widetilde \gamma\in SL(n+1,\C)$ be a lift of $\gamma$ and   $(k, \{V_j\}_{j=1}^k, \{\beta_j\}_{j=1}^k, \{\lambda_j\}_{j=1}^k, \{\gamma_j\}_{j=1}^k)$ be a  block decomposition for $\widetilde \gamma$. If $k\geq 2$ and  for each $j\in \{1,\ldots, k\}$ it is verified that 
	$\gamma_j$ is a $l_j\times l_j$ Jordan block with proper value 1, then 
	\[
	\Omega_{Kul}(\langle \gamma\rangle)=\Bbb{P}^n_{\Bbb{C}}\setminus\left \langle \left \langle V_1\cup\bigcup_{j=2}^k \Lambda_{Kul}(\langle\gamma_j\rangle )\right \rangle \right \rangle.
	\] 
\end{lemma}

Now the following are   straightforward Corollaries whose proof we omit here.

\begin{corollary}\label{c:parmax}
	Let $\gamma\in PSL(n+1,\Bbb{C})$ be a parabolic element, $\widetilde \gamma\in SL(n+1,\Bbb{C})$ a lift  of $\gamma$ and $( \{V_j\}_{j=1}^k,\{\beta_j\}_{j=1}^k, \{T_j\}_{j=1}^k)$ a block decomposition of $\widetilde{\gamma}$, then
	\begin{enumerate}
\item The largest open set $\Omega_\gamma$ on which $\langle \gamma\rangle$ acts properly discontinuously is 
\[
\Omega_\gamma= \Bbb{P}^n_{\Bbb{C}}\setminus
 \langle\langle V_1\cup \bigcup_{j=2}^k\{v_{j1},\ldots,v_{j,ceiling(dim v_j/2)} \} \rangle \rangle 
\]
\item We have $Eq(\langle\gamma\rangle)= \Omega_{Kul} (\langle\gamma\rangle)$ if and only if for each $j,m>2$  we have $dim V_j=dim V_m$.
\item We have $ \Omega_{Kul} (\langle\gamma\rangle)$ is the largest open set on which  $\langle\gamma\rangle $ acts properly discontinuously  if and only if for each $j>2$  we have $dim V_j=2,3$.
\item We have $ Eq(\langle\gamma\rangle)=\Omega_{Kul} (\langle\gamma\rangle)$ is the Largest open set on which  $\langle\gamma\rangle $ acts properly discontinuously  if and only if 
for each $j,m>2$  we have $dim V_j=dim V_m$ and $dim V_2=2,3$.

		\end{enumerate}
	
\end{corollary}

\begin{corollary}\label{c:parmax}
	Let $\gamma\in PSL(n+1,\Bbb{C})$ be a parabolic element, $\widetilde \gamma\in SL(n+1,\Bbb{C})$ a lift  of $\gamma$ and $( \{V_j\}_{j=1}^m,\{\beta_j\}_{j=1}^m, \{T_j\}_{j=1}^m)$ a block decomposition of $\widetilde{\gamma}$. If $\gamma$ is conjugate to an element in $PU(k,l)$, then:
	\[
\sum_{j=2}^m ceiling((dim v_j)/2)	\leq k.
	\]
	
\end{corollary}

As we are going to see the  results developed in this section will provide a good understanding of the dynamic  for loxodromic  elements.

\section{Loxodromic Transformations} \label{s:lox}

Recall that a loxodromic element $\gamma$ in $\PSL(2	,\Bbb{C})$ by definition has two fixed points in $p,q\in\P^1$. One of these points is repelling, the other attracting. Due to this fact 
one  can always choose a small enough ball $W$ with center at the attracting point 
such that $\gamma ( \overline W)  \subset W$. We will  see bellow that this property  characterizes the loxodromic elements.  The following technical lemmas will help in this  characterization of the loxodromic elements.

\begin{lemma}\label{l:t1}
	Let  $\gamma \in  {\rm PSL} (n+1 , \mathbb C)$ be an element for which   there is a proper    open set   $W$  in $\P^{n} $ such that    $\gamma ( \overline W ) \subset W$, then
	$ \Lambda(\langle\gamma\rangle )$ is a non-empty disconnected set.
\end{lemma}

\begin{proof}  Define 
	$$\Omega =\bigcup_{n\in \Bbb{Z}}  \gamma^n(\overline{W}\setminus \gamma(\overline{W})).$$
	Then $\Omega $ is a non-empty set where $\langle \gamma \rangle $ acts properly discontinuously.   In consequence 
	$$
	\Lambda(\langle \gamma \rangle)\subset
	W 
	\cup
	\P^n \setminus \overline  W.
	$$
	To conclude observe that   $ \Lambda(\gamma)\cap W\neq \emptyset$ and $\Lambda(\gamma)\cap \P^n\setminus \overline W\neq \emptyset$.
\end{proof}



  \begin{lemma}[See   \cite{shaw}] \label{l:t2}
	Let $T \in SL(n+1,\Bbb{C})$  and $1 \leq k \leq n+1$. If $\alpha_1,\ldots, \alpha_{n+1} $ are the  eigenvalues of $T$, then the eigenvalues of $\bigwedge_ k T$
	has the form $\alpha_{j_1}\cdots \alpha_{j_k }$ where $j_1,\ldots , j_k\in \{1,\ldots, n+1\}$ and $j_l<j_k$, whenever $k<l$.
\end{lemma}

\begin{lemma} \label{l:t3}
	Let $T \in SL(n+1,\Bbb{C})$ and  $\alpha_1,\ldots, \alpha_{n+1} $ be  the eigenvalues of $T$. If $p\in \Bbb{C}^{n+1}$ is a  eigenvector for  $\alpha_1$, $\vert \alpha_1\vert=\max \{\vert \alpha_j\vert:j \in \{1,\ldots, n+1\}\} $ and  $\vert \alpha_1\vert \neq \vert \alpha _k\vert  $ for $k>2$, then $[p]$ is an attracting fixed point for the action  of $ [[T]]$ in $\Bbb{P}^n_{\Bbb{C}}$
\end{lemma}
\begin{proof} By the Normal Jordan normal form theorem, there is $v\in \C^{n+1}$ and $W$ a linear subspace such that $\C^{n+1}=\langle v\rangle\oplus W$, $W$ is $T$-invariant, $v$ is  eigenvector with  eigenvalue $\alpha_1$ and  the spectral radius of $T\mid_W$ is less than $ \vert \alpha_1\vert$. Consider the affine chart $ w\in W \mapsto [v+w]$,  in this chart  $[v]$ correspond to the origin and $[[T]]$ is simply $\alpha_1^{-1}T\mid_W$. Since $D(\alpha_1^{-1}T\mid_W)(0)=\alpha_1^{-1}T\mid_W$ and the spectral radius of  $\alpha_1^{-1}T\mid_W$ is less than $1$, we conclude that $[v]$ is an attracting fixed point. 
\end{proof}

\begin{lemma} \label{l:t4}
	Let $U\subset \iota Gr(k,n)$ be an open set (resp. a closed set), then $\bigcup_{\ell\in U}\ell$  is an open set in $\P^n$  (resp. closed).  
\end{lemma}
\begin{proof}Let us show the case when $U$ is an open set. Clearly, is enough to  assume that  $U$ is an open ball in $\iota Gr(k,n)$. Let us assume that $U=B_{ d}(r,\ell)$, where $d$ is the Fubini-Study metric restricted to $\iota Gr(k,n)$. Let $v_1,\ldots, v_{k+1}\in \C^n\setminus\{0\}  $ be points in general position such that $[Span( v_1,\ldots, v_{k+1}) \setminus\{0\} ]=\ell$, for each set $W=\{w_1,\ldots, w_k\}\subset  \C^{n}$ of points in general position, consider the following  function $\nabla_W:\P^n\setminus [\langle W\rangle\setminus \{0\} ]\rightarrow \Bbb{R}^+$, given by 
	\[
	\nabla_W([z])= d( [w_1\wedge \cdots \wedge w_k\wedge z],[ v_1\wedge \cdots \wedge v_{k+1}])
	\]
	Clearly $\nabla_W$ is a well defined continuous function. To conclude is enough to observe  that 
	$\bigcup_{\ell\in U}\ell=\bigcup_W\{z\in  \P^n\setminus [\langle W\rangle\setminus \{0\} ]: \nabla_W(z)<r\}$. \\
	
	Finally, let us proof the case when $U$ is a closed   set. Let $(x_m)\subset \bigcup_{\ell\in U}\ell$ be a sequence converging to $x$. For each $m$ we can choose an element $\ell_m\in U$ such that $x_m\in \ell_m$. Since $Gr(k,n)$ is compact we can assume that there is $\ell_0\in U$ such that $\ell_m \xymatrix{ \ar[r]_{m \rightarrow  \infty}&} \ell_0 $, in the topology of $Gr(k,n)$. In consequence $\ell_m \xymatrix{ \ar[r]_{m \rightarrow  \infty}&} \ell_0 $ as closed sets in the Hausdorff topology. Therefore $x\in \ell_0$, which concludes the proof. 
\end{proof}

\begin{lemma} \label{l:t5}
	Given $\mathcal{L}\in Gr(k,n)$, there is  an open set $U\subset Gr(k,n)$ such that $\mathcal{L}\in U$ and $\bigcup_{\ell\in U}\ell \neq \P^n$. 
\end{lemma}
\begin{proof}
	Let $W\subset \P^n\setminus \mathcal{L}$ a non empty open set  such that $\overline W\cap \mathcal{L}=\emptyset$. Define $$\widetilde W=\{\ell\in Gr(k,n): \ell\cap \overline{W}\neq \emptyset \},$$ clearly $\widehat W=Gr(n,k)\setminus \widetilde W$ is an open set  containing $\mathcal{L}$ also  satisfying 
	$W\subset \P^n \setminus\bigcup_{\ell\in \widehat W} \ell $. Which concludes the proof.
\end{proof}

\subsection*{Proof of Theorem \ref{t:l1}}
\begin{proof} 
Let $\gamma\in SL(n+1,\C)$ be a linear transformation with one non-unitary  eigenvalue, then  by the Normal Jordan form we can assume that 
$\gamma$ can be written as 
\[
\gamma=
\left (
\begin{array}{llll}
r_1 A_1 &          \\
        & r_2A_2 &\\
        &        & \ddots &\\
        &        &        & r_kA_k
\end{array}
\right )
\]
where $r_k<r_{k-1}<\ldots< r_1 $ and each matrix $A_k$ has only unitary  eigenvalues. Now let   $\widetilde n=dim A_1$ and $\alpha_1,\ldots, \alpha_{\widetilde n}$ be the  eigenvalues of 
$A_1$. A straightforward calculation shows that  $p=e_1\wedge \cdots\wedge e_{\widetilde n}$ is an eigenvalue of $T=\bigwedge^{\tilde n} \gamma $ with eigenvalue $\alpha=r_1^{\widetilde n}\alpha_1\cdots \alpha_{\widetilde n}$, moreover by Lemma \ref{l:t2} we deduce that  $\alpha$ is a simple root of $Det(T-\lambda I)=0$ and    $r_1^{\tilde n}=\max \{\vert \beta\vert:\beta \textrm{ eigenvalue of }T\} $. By Lemma \ref{l:t3} it yields that $[p]$ is an attracting  fixed point of $[[T]]$ acting on $P(\bigwedge^{\tilde n} \Bbb{C}^n)$. Due to the Plucker embedding  and Lemma \ref{l:t5} we conclude that  there is an open set $U\subset Gr(\widetilde n-1,n )$ such that: $[[\gamma] ](\overline U)\subset U$, 
$ \langle\langle e_1,\ldots, e_{\widetilde n}  \rangle\rangle \in U$ and $\bigcup_{\ell\in U}\ell \neq \P^n$. To conclude observe that Lemma \ref{l:t4} yields that   $W=\bigcup_{\ell\in U}\ell$ is a proper open set which satisfy $[[\gamma]](\overline W)\subset W$. Which concludes this part of the proof. 

The proof now concludes by  Lemmas \ref{l:lila} and \ref{l:t1}.
\end{proof}

\begin{theorem}\label{t:kpar2}
Let  $\gamma\in SL(n+1,\C)$ be a given by:
\[
\left (
\begin{array}{llll}
r_1 A_1 &          \\
        & r_2A_2 &\\
        &        & \ddots &\\
        &        &        & r_kA_k
\end{array}
\right )
\]
where  $0<r_k<r_{k-1}<\ldots< r_1 $ and each  $A_j$ has only unitary eigenvectors, also let  
\[
W_1=
\left \{
\begin{matrix}
\langle \langle \Lambda_{Kul} (\langle A_1\rangle )\cup \{ e_{dim A_1+1},\ldots, e_{n+1}\} \rangle \rangle & \textrm{ if } A_1 \textrm{ is non-diagonalizable.}\\
\langle \langle e_{dim A_1+1},\ldots, e_{n+1} \rangle \rangle & \textrm{ if } A_1 \textrm{ is diagonalizable}
\end{matrix}
\right.
\]
and
\[
W_k=
\left \{
\begin{matrix}
\langle \langle \Lambda_{Kul} (\langle A_k\rangle )\cup \{  e_{dim 1},\ldots, e_{\sum_{j=1}^{k-1} dim(A_j)}\} \rangle \rangle & \textrm{ if } A_k \textrm{ non- diagonalizable.}\\
\langle \langle e_{dim 1},\ldots, e_{\sum_{j=1}^{k-1} dim(A_j)} \rangle \rangle & \textrm{ if } A_k \textrm{ is diagonalizable}
\end{matrix}
\right.
\]
Thus  $\Lambda_{Kul}(\langle \gamma\rangle)=W_1\cup W_k$.
\end{theorem}
\begin{proof}
	The following proof is a  straightforward application  of Lemmas  \ref{l:par11},  \ref{l:par22} and  \ref{l:par22a}. 
	\end{proof}
As in the two dimensional case,  see \cite{kulkarni}, the Kulkarni's discontinuity region is not the largest open set where the cyclic   groups act properly discontinuously, the following  example depicts this situation.

\begin{example} \label{e:loxmax}
Let $\gamma\in PSL(n+1,\C)$ be a loxodromic element such that there is     $\widetilde \gamma\in SL(n+1,\C)$ a lift of $\gamma$ and  $(k,\{V_i\}_{i=1}^k,\{\gamma_i\}_{i=1}^k,\{r_i\}_{i=1}^k)$  an unitary decomposition of $\widetilde \gamma$ such that $k>3$, then:
\[
\Omega_1=\P^n\setminus \left ([V_1\setminus \{0\}]\cup \left [\langle\langle \bigcup_{j>2} V_j \rangle\rangle \setminus \{0\}\right ] \right)
\]
\[
\Omega_2=\P^n\setminus 
\left(
[
V_k\setminus \{0 \}  ]\cup \left [\langle\langle \bigcup_{j<k} V_j \rangle\rangle \setminus \{0\}\right ] \right )
\]
are maximal discontinuity regions. 
\end{example}

 In the two dimensional setting is well known that loxodromic transformation have at most 2 maximal open sets on which the respective cyclic group acts properly discontinuously, in the higher dimensional case this situation can be more complicated as the following example shows.
  
\begin{example}
Let us consider the element $\gamma\in PU(1,n)$ induced by the  following matrix
\[
\widetilde\gamma
=
\left (
\begin{array}{lll}
2 &             &\\
& I_{n-1} &\\
&             &2^{-1}
\end{array}
\right )
\]
and $\Gamma=\langle\langle \gamma\rangle \rangle$.
It is not hard to show that 
$$Eq(\Gamma)=\Bbb{P}^n\setminus (\{e_1\}^\bot\cup \{e_{n+1}\}^\bot )=\Omega_{Kul}(\Gamma).$$
On the other hand,  when one tries to determine if the previous sets are  maximal  open sets on which  $\Gamma$ acts properly 
discontinuously, we find the following phenomena: given $U,W\subset Span(\{e_2,\ldots, e_n\})=\mathcal{L}$ disjoint open sets such 
that $\overline U\cup \overline W=\mathcal{L}$, define 
\[
\begin{array}{l}
\mathcal{U}=\{\overleftrightarrow{e_1,u}\mid u\in \overline U\}\\
\mathcal{V}=\{\overleftrightarrow{e_{n+1},v}\mid v\in \overline V\}.\\
\end{array}
\]
Is not hard to show $\Bbb{P}^n\setminus(\mathcal{U}\cup \mathcal{V})$ is a maximal open set on which $\Gamma$ acts properly discontinuously and every maximal open set for the action of $\Gamma$ arises in this way.
\end{example}

 From the previous section one expect to get examples where the Kulkarni and the equicontinuity  agree  and others where they  disagree. This situation is showed by the following  examples.
 
\begin{example}\label{e:ekdl1}
Consider first  the matrix $S$:
\begin{displaymath}
S=
\left (
\begin{array}{ll}
rB& 0\\
0 & r^{-1}B \\
\end{array}
\right )
\end{displaymath}
where $r>1$ and $B$ is the $k\times k$  identity.
Then one can check
  $Eq(\langle S \rangle)=\Omega_{\rm Kul} (\langle S \rangle)$.
  \end{example}
  \begin{example} \label{e:ekdl2}
   Let 
\begin{displaymath}
C=
\left (
\begin{array}{cccc}
rB\\
        & rD\\
        &     &r^{-1}D \\
        &     &             & r^{-1}B \\
\end{array}
\right )\;,
\end{displaymath}
where $r>1$,  $B$ is a $k+1\times k+1$-Jordan block and  $D$ is a $k\times k$-Jordan block.
Is clear that  $\Omega_{\rm Kul} (\langle C \rangle)\neq Eq(\langle C \rangle).$
\end{example}

\subsection{ Existence of Loxodromic elements} 
Loxodromic elements will play an important role in the dynamic of kleinian groups, in this final part of the article we are going to show that in the case of strongly irreducible groups we can ensure the existence of loxodromic elements.

\begin{lemma} \label{l:eloxy}
	Let $(\gamma_m)\subset PSL(n+1,\Bbb{C})$ and $\gamma\in PS(n+1,\Bbb{C})\setminus PSL(n+1,\Bbb{C})$ such that $\gamma_m  \xymatrix{ \ar[r]_{m \rightarrow  \infty}&}\gamma$ as pseudo projective transformations. If $Im(\gamma)\cap Ker(\gamma)=\emptyset$, then for $m$ large we get $\gamma_m$ is loxodromic.  
\end{lemma}
\begin{proof}
	We have that $\gamma_m \xymatrix{ \ar[r]_{m \rightarrow  \infty}&}\gamma$ uniformly on compact sets of $\Bbb{P}^n_{\Bbb{C}}\setminus Ker(\gamma)$. Let $W$ be an open neighborhood  of $Im(\gamma)$ such that $\overline W\subset \Bbb{P}^n_{\Bbb{C}}\setminus Ker(\gamma)$, then for $m$ large  we get $\gamma_m(\overline W)\subset W$, that is $\gamma_m$ is loxodromic.
\end{proof}

\begin{lemma} \label{l:p}
	Let $(M_m)_{m\in \N}, (N_m)_{m\in \N}\subset M(n,\C) $ be sequences of matrices such that $M_m \xymatrix{ \ar[r]_{m \rightarrow  \infty}&} M$ and $N_m\xymatrix{ \ar[r]_{m \rightarrow  \infty}&}
	N$  point-wise. If $Im (N) \nsubseteq Ker( M)$, then $[M_mN_m]\xymatrix{ \ar[r]_{m \rightarrow  \infty}&}
	[MN]$ in the sense of pseudo-projective transformations.
\end{lemma}
\begin{proof}
	By the continuity of the  matrices product we deduce that  $M_mN_m\xymatrix{ \ar[r]_{m \rightarrow  \infty}&}
	MN$, therefore  in order to conclude the proof we need to show that $MN\neq 0$, which is trivial since $Im(N)\nsubseteq Ker(M)$.
\end{proof}

From linear algebra we know that for every linear transformation $T:\C^n\rightarrow \C^n$ and every $1\leq k\leq n$ we got a linear transformation $\bigwedge_k T:\bigwedge^k_{j=1}\C^n\rightarrow \bigwedge^k_{j=1}\C^n$ which is induced by $\bigwedge_k T(v_1\wedge \cdots \wedge v_k)= T(v_1)\wedge \cdots\wedge T( v_k) $, so if we have a subgroup $\Gamma\subset PSL(n+1)$, we can define

\[
\bigwedge_k \Gamma=\left\{\left [\left  [\bigwedge_k \gamma\right]\right ]: [[\gamma]]\in \Gamma\right \}
\] .

\begin{lemma}\label{l:w}
	Let $(M_m)_{m\in \N}\subset M(n,\C) $ be  a sequence of matrices such that $M_m \xymatrix{ \ar[r]_{m \rightarrow  \infty}&} M$  point-wise.
	\begin{enumerate}
		\item\label{i:lw1} If $dim (Im (M) )\geq k$, then $\left [\bigwedge^k_{j=1}M_m\right ]\xymatrix{ \ar[r]_{m \rightarrow  \infty}&}
		\left [\bigwedge^k_{j=1}M \right ]$ in the sense of pseudo-projective transformations.
		\item \label{i:lw2} If $dim (Im (M) )= k$, then: 
		\[
		Im\left (	\left [\bigwedge^k_{j=1}M \right ]\right )=\iota_{k-1,n-1}(Im(M)).
		\]
		Moreover $$Im \left (\left [ \left [\bigwedge^k_{j=1}M \right ]\right] \right )\subset Ker \left (\left [ \left [\bigwedge^k_{j=1}M \right ]\right] \right )$$ if and only if $Ker([M])\cap Im([M])\neq \emptyset$.
	\end{enumerate}
\end{lemma}
\begin{proof} Let us show part (\ref{i:lw1}).
	By the continuity of wedge product we have 
	$$\bigwedge^k_{j=1} M_m\xymatrix{ \ar[r]_{m \rightarrow  \infty}&}\bigwedge^k_{j=1} M,$$ therefore  in order to conclude the proof we need to show that 
	$\bigwedge^k_{j=1} M\neq 0$. 
	Let  $v_1, \ldots, v_k \in Im(M)$ be linearly independent vectors and 
	$\widetilde v_1,\ldots, \widetilde v_k$ be such that 
	$M(\widetilde v_i)= v_i$, then 
	\[
	M_m( \widetilde v_1 \wedge \cdots \wedge \widetilde v_k )\xymatrix{ \ar[r]_{m \rightarrow  \infty}&}
	M( \widetilde v_1 \wedge \cdots \wedge \widetilde v_k )= v_1 \wedge \cdots \wedge v_k  \neq 0,
	\]
	which concludes the proof.\\
	
	The proof of part (\ref{i:lw2}) goes as follows: let $W\in Gr(k-1,n-1)$, then we can choose $w_1,\ldots, w_k\in W$  so $\iota_{k-1,n-1}( W)=[w_1\wedge\cdots\wedge  w_k]$, in consequence 
	
	\[
	\left	[\left [\bigwedge_k M \right ]\right ]
	[w_1\wedge\cdots\wedge w_k]=
	\left \{
	\begin{matrix}	
	\iota_{k-1,n-1}(Im(M)) & \textrm{ if } W\cap Ker (M)=\emptyset \\
	0 & \textrm{ in other case}
	\end{matrix}
	\right.
	\]
	which concludes the proof.
\end{proof}


\begin{lemma} \label{l:hypg}
	Let $\mathcal{L}\subset Gr(n-1,n)$ be an infinite  collection  whose elements are in general position and $\mu:\mathcal{L}\rightarrow \Bbb{P}^{n}_{\Bbb{C}}$ be such that $\mu(\ell)\in \ell$  for every $\ell\in \mathcal{L}$, then there are $\mathcal{K}_1,\mathcal{K}_2\in \mathcal{L}$  such that  
	$\mu( \mathcal{K}_2)\notin \mathcal{K}_1 $ and $\mu( \mathcal{K}_1)\notin \mathcal{K}_2$.
\end{lemma}
\begin{proof} The proof is going to be by induction on $n$. For $n=2$ the proof is worked out in \cite{BCN1}. Now let us assume that the result is valid for $n_0$ but is not longer true for $n_0+1$. Then there    $\mathcal{L}\subset Gr(n,n+1)$ and $\mu:\mathcal{L}\rightarrow{P}^{n+1}_\Bbb{C}$as in the hypothesis of this result such that  for every  $\mathcal{K}_1,\mathcal{K}_2\in \mathcal{L}$  we have   either 
	$\mu( \mathcal{K}_2)\in \mathcal{K}_1 $ or $\mu( \mathcal{K}_1)\in \mathcal{K}_2$.
	Since the elements of  $\mathcal{L}$ are in general position we can assume that $\mu$ is injective. Let $\mathcal{K}_0\in \mathcal{L}$ be a fixed, then there exist $\mathcal{L}^\prime\subset \mathcal{L}$ an infinite subset such that: $$\mu(\mathcal{K}_0)\notin \bigcup \mathcal{L}^\prime.$$ 
	Thus  
	$\mathcal{L}^\prime \cap \mathcal{K}_0=\{\ell\cap \mathcal{K}_0:\ell\in \mathcal{L}^\prime\}$ is an infinite set of projective subspaces with  co dimension one on  $\mathcal{K}_0$ and  lying  in general position. Moreover, if  $$\mu\prime(\ell\cap \mathcal{K}_0)=\mu(\ell)
	\textrm{ for every } \ell\in \mathcal{L}^\prime,$$ we get $\mu\prime(\ell\cap \mathcal{K}_0) \in \mathcal{K}_0\cap\ell$.  By the inductive hypothesis, there are $\mathcal{K}_1,\mathcal{K}_1,\in\mathcal{L}^\prime $ such that 
	\[
	\begin{matrix}
	\mu^\prime (\mathcal{K}_1\cap\mathcal{K}_0)\notin \mathcal{K}_0\cap \mathcal{K}_2\\
	\mu^\prime (\mathcal{K}_2\cap\mathcal{K}_0)\notin \mathcal{K}_0\cap \mathcal{K}_1\\
	\end{matrix}
	\]
	therefore $\mu(\mathcal{K}_1)\notin \mathcal{K}_2$ and $\mu(\mathcal{K}_2)\notin \mathcal{K}_2$, which is a contradiction. This concludes the proof.
\end{proof}
\begin{corollary}
	Let $\Gamma\subset PSL(n+1,\Bbb{C})$ be a discrete strongly irreducible group, then $\Gamma$ contains a loxodromic element.	
\end{corollary}
\begin{proof}  
	Since $\Gamma$ is discrete we conclude that there is $\gamma\in \Gamma$  with infinite order, see \cite{CSN}. Let assume that $\gamma$ is parabolic then there a strictly increasing sequence $(n_m)\subset  \Bbb{N}$ and $\rho_0\in PS(n+1,\Bbb{C})\setminus PSL(n+1,\Bbb{C})$ such that $\rho_m=\gamma^{n_m} \xymatrix{
		\ar[r]_{m \rightarrow \infty}&} \rho_0 $ as pseudo projective transformations and $Im(\rho
	)\subset Ker(\rho)$.  Let $\widetilde \rho_m\in M(n+1,\Bbb{C})$ be a lift of $\rho_m$ such that $\widetilde \rho_m \xymatrix{	\ar[r]_{m \rightarrow \infty}&} \widetilde \rho_0 $ point-wise Set $k=dim(Im(\rho_0))+1$, by Lemma \ref{l:w}, we deduce that $ \bigwedge^k \rho_m=[[\bigwedge ^k\widetilde \rho_m] ]\xymatrix{	\ar[r]_{m \rightarrow \infty}&}\bigwedge^k \rho_0=[[\bigwedge ^k\widetilde \rho_0] ]$ and $Im(\bigwedge^k \rho_0)$ is a single point contained in $Ker(\bigwedge^k \rho_0)$. Since the  action  of $\bigwedge_k \Gamma$ acts strongly irreducible we deduce that there is a sequence of distinct elements   $(\tau_m)\subset \bigwedge_k \Gamma$ such that 
	$\mathcal{L}=\{\tau_m(Ker(\bigwedge^k \rho_0)):m\in \Bbb{N} \}$  is a family of  hyperplanes in general position. By applying Lemma \ref{l:hypg} to $\mathcal{L}$ and $\mu:\mathcal{L}\rightarrow \Bbb{P}(\bigwedge^k \Bbb{C}^{n+1})$ given by 
	$$\mu(\tau_m(Ker(\bigwedge^k \rho_0)))=\tau_m(Im(\bigwedge^k \rho_0)),$$ we deduce that  there are $i_0, j_0\{1,\ldots, dim (\bigwedge ^k \Bbb{C}^{n+1}) \}$ satisfying:
	\[
	\begin{matrix}
	\tau_{i_0}(Im(\bigwedge^k \rho_0))\notin \tau_{j_0}(Ker(\bigwedge^k \rho_0));\\
	\tau_{j_0}(Im(\bigwedge^k \rho_0))\notin \tau_{i_0}(Ker(\bigwedge^k \rho_0)).\\
	\end{matrix}
	\]
	If $\widetilde\tau_{j_0},\widetilde\tau_{i_0}\in End(\bigwedge^k \Bbb{C}^{n+1})$ are  lifts of $\tau_{j_0}$ and $\tau_{i_0}$ respectively, then by Lemma \ref{l:p} we conclude that 
	\begin{scriptsize}
		\[
		\eta_m=\tau_{i_0}^{-1}\tau_{j_0} \left(\bigwedge^k	\rho_m\right )\tau_{j_0}^{-1}\tau_{i_0}\left (\bigwedge^k\rho_m  \right )\xymatrix{
			\ar[r]_{m \rightarrow \infty}&} \eta_0= \left [\left [ \widetilde\tau_{i_0}^{-1}\widetilde \tau_{j_0} \left(\bigwedge^k	\widetilde \rho_0\right )\widetilde \tau_{j_0}^{-1}\widetilde\tau_{i_0}\left (\bigwedge^k\widetilde\rho_0 \right )\right ]	\right ],
		\]
	\end{scriptsize}
	Observe
	\begin{equation}\label{e:kernel}
	Ker( \eta_0)
	=
	\left [
	\left(
	\bigwedge_k\widetilde{\rho_0}
	\right )
	^{-1}
	\left (
	\widetilde\tau_{i_0}^{-1}
	\left (
	\widetilde\tau_{j_0}
	\left (
	Ker
	\left (
	\widetilde{\bigwedge_k \rho_0}
	\right )
	\right )
	\right )
	\right )
	\setminus\{0\}
	\right ]
	\end{equation}
	Since $	Im(\bigwedge^k \rho_0)\notin\tau_{i_0}^{-1} \tau_{j_0}(Ker(\bigwedge^k \rho_0))
	$ by equation \ref{e:kernel} we deduce that 
	\[	
	\begin{matrix}
	Ker( \eta_0)
	=Ker(\bigwedge_k\rho_0)\\ 
	Im(\eta_0)=\tau_{i_0}^{-1} \tau_{j_0}(\bigwedge_k\rho_0)
	\end{matrix}
	\]
	Thus $Im(\eta_0)$ is a point not contained in $Ker(\eta_0)$, by Lemma \ref{l:eloxy} we conclude that for $m$ large 
	$\eta_m$ is loxodromic, if $\eta\in \Gamma$ satisfies that $\bigwedge_k \eta=\eta_m$ we conclude that $\eta$ is loxodromic, which concludes the proof.
\end{proof}

\section{Final toughs}\label{s:main}

 In the case of projective parabolic transformation, our previous discussion shows that for elements in $PU(k,l)$, $k\geq 2$, the Kulkarni's  discontinuity set is not longer the largest open set where the corresponding group acts properly discontinuously. For discrete groups in $PU(k,l)$ it would be interesting to determine if there is a generalization of the notion of the Chen-Greenberg limit set, determine   its limit set in the Kulkarni's sense and if there is an open set which  is the largest open set on which the group acts properly  discontinuously. 
  
 In the one and  two dimensional case and for some transformations in $PU(k,l)$, maps can be classified by the use of the trace,  see \cite{CSN, kris},  it would be useful to have a similar result.

 Let $X$ be   the space,
 of all positive definite, symmetric $3 \times  3$-matrices with real coefficients,
of determinant 1, there is a metric $d$ such that $X$  is a $CAT(0)$-space and  the action of $SL(3,\R)$  on $X$ given by  $fxf^{t}$, where $x\in X$ and $f\in SL(3,\R)$, is by isometries. Then by using the classification of isometries in $CAT(0)$-spaces one can classify elements of $SL(3,\R)$ in to parabolic, loxodromic and elliptic  however is not hard to show, see \cite{fns}, that such classification does no agree with the one induced by definition \ref{d:tipos}.

 From the the theory of $CAT(0)$-spaces one know that isometries can be classified in to three types namely elliptic, parabolic or hyperbolic. In virtue of the similarity of our results with the ones coming from $CAT(0)$-spaces,  is natural to ask if it is possible to use this theory to deal with the classification of projective transformations. We got two naive partial answers: first since $\P^n$ is  compact one cannot use directly the theory of $CAT(0)$-spaces   to deal with the problem of classification of projective transformations, second, a result in \cite{fns} asserts that the fixed set of parabolic elements should be contractible in the Tits boundary of $X$, in consequence for $n>1$, the projective space  $\P^n$ cannot be the tits boundary of a proper $CAT(0)$-space where the projective transformations are extensions of isometries of $X$.

\section*{Acknowledgments}
The authors would like to thank to people the  IMUNAM Unidad Cuernavaca as well as the IIT of UACJ for
 their   hospitality and kindness during the writing of this paper and we are specially grateful to N. Gusevskii,  J. P. Navarrete, John Parker, Jos\'e 
 Seade  and A. Verjovsky  for valuable discussions.

\bibliographystyle{amsplain}

\end{document}